\documentclass[12pt]{amsart}
\usepackage{amssymb,latexsym,mathtools}
\usepackage[margin=1.25in]{geometry}
\usepackage{mathrsfs}
\usepackage{mathtools}
\usepackage{graphicx}
\usepackage{esint}
\usepackage{braket}
\usepackage{url, hyperref}
\newtheorem{theorem}{Theorem}
\newtheorem{lemma}{Lemma}
\newtheorem{definition}{Definition}

\newtheorem{prop}{Proposition}

\newtheorem{remark}{Remark}
\newtheorem{corollary}{Corollary}

\begin{document}

\title[Allen-Cahn Equation with Periodic Mobility]{Homogenization of the Allen-Cahn Equation with Periodic Mobility}
\author[P.S.\ Morfe]{Peter S.\ Morfe}

\maketitle

\begin{abstract}  We analyze the sharp interface limit for the Allen-Cahn equation with an anisotropic, spatially periodic mobility coefficient and prove that the large-scale behavior of interfaces is determined by mean curvature flow with an effective mobility.  Formally, the result follows from the asymptotics developed by Barles and Souganidis for bistable reaction-diffusion equations with periodic coefficients.  However, we show that the corresponding cell problem is actually ill-posed when the normal direction is rational.  To circumvent this issue, a number of new ideas are needed, both in the construction of mesoscopic sub- and supersolutions controlling the large-scale behavior of interfaces and in the proof that the interfaces obtained in the limit are actually described by the effective equation.  \end{abstract}\

\section{Introduction}  

\subsection{Overview}  In this paper, we study the sharp interface limit of the Allen-Cahn equation with periodic mobility.  Given $u_{0} \in UC(\mathbb{R}^{d}; [-1,1])$, we are interested in the behavior as $\epsilon \to 0^{+}$ of the solutions $(u^{\epsilon})_{\epsilon > 0}$ of the following equations:
	\begin{equation} \label{E: AC mobility}
		\left\{ \begin{array}{r l}
			m(\epsilon^{-1} x, \epsilon Du^{\epsilon}) u^{\epsilon}_{t} - \Delta u^{\epsilon} + \epsilon^{-2} W'(u^{\epsilon}) = 0 & \text{in} \, \, \mathbb{R}^{d} \times (0,\infty), \\
			u^{\epsilon} = u_{0} & \text{on} \, \, \mathbb{R}^{d} \times \{0\}.
		\end{array} \right.
	\end{equation}
Here $W : [-1,1] \to [0,\infty)$ is a double well-potential with wells of equal depth (for concreteness, $W^{-1}(\{0\}) = \{-1,1\}$) and $m : \mathbb{T}^{d} \times \mathbb{R}^{d} \to (0,\infty)$ is a mobility coefficient.

Equation \eqref{E: AC mobility} describes the large-scale behavior of phase transitions when the free energy is determined by the Allen-Cahn functional and energy dissipation is determined by $m$.  The scaling results from blowing up space by $\epsilon^{-1}$ and time, by $\epsilon^{-2}$.  Without going into details yet, we only assume that $W$ and $m$ are bounded and sufficiently smooth.

Classically, in the spatially homogeneous, isotropic setting (i.e.\ when $m \equiv 1$), the large-scale behavior of $(u^{\epsilon})_{\epsilon > 0}$ is described by the so-called \emph{sharp interface limit}.  Informally, this means that, as $\epsilon \to 0^{+}$, $u^{\epsilon}(\cdot,t) \to 1$ in $\Omega_{t}$ and $u^{\epsilon}(\cdot,t) \to -1$ in $\mathbb{R}^{d} \setminus \Omega_{t}$, where the family of open sets $(\Omega_{t})_{t \geq 0}$ is a mean curvature motion with $\Omega_{0} = \{u_{0} > 0\}$.

In the periodic setting considered here, we prove below that a similar result holds, except that the effective interface velocity $V$ has the following form:
	\begin{equation} \label{E: velocity}
		\overline{m}(n_{\partial \Omega_{t}}) V = \kappa_{\partial \Omega_{t}}.
	\end{equation}
Here $n_{\partial \Omega_{t}}$ and $\kappa_{\partial \Omega_{t}}$ are the normal vector and mean curvature of the surface $\partial \Omega_{t}$ and $\overline{m}$ is the effective mobility determined by $m$ and $W$ through averaging effects.

\subsection{Motivation}  Since the '90s, there has been significant interest in whether or not the main features of bistable reaction-diffusion equations in spatially homogeneous media persist in the periodic setting.  Beginning with Barles and Souganidis \cite{barles souganidis}, some attention has been devoted to the sharp interface limit in periodic media, but the question is largely open.  

In \cite{barles souganidis}, the authors consider equations of the following form in the special case when $m \equiv 1$:
	\begin{equation} \label{E: general gradient flow}
		m(\epsilon^{-1} x, \epsilon Du^{\epsilon}) u_{t}^{\epsilon} - \text{div}(a(\epsilon^{-1} x) Du^{\epsilon}) + \epsilon^{-2} W'(u^{\epsilon}) = 0 \quad \text{in} \, \, \mathbb{R}^{d} \times (0,\infty).
	\end{equation}
Notice that this equation is, at least formally, the gradient flow of the spatially heterogeneous free energy
	\begin{equation} \label{E: functional}
		\mathcal{F}^{a}(u; \Omega) = \int_{\Omega} \left( \frac{1}{2} \langle a(y) Du(y), Du(y) \rangle + W(u(y)) \right) \, dy
	\end{equation}
with respect to the $L^{2}$-Riemannian metric given by
	\begin{equation} \label{E: metric}
		\langle v^{\epsilon}, w^{\epsilon} \rangle_{u^{\epsilon}} = \int_{\mathbb{R}^{d}} v^{\epsilon}(x) w^{\epsilon}(x) m(\epsilon^{-1} x, \epsilon Du^{\epsilon}(x)) \, dx.
	\end{equation}
Here we follow Taylor and Cahn \cite{taylor cahn}.

Barles and Souganidis prove that, under some strong assumptions, the large-scale behavior of \eqref{E: general gradient flow} is described by an anisotropic curvature flow in the sharp interface limit (see \cite[Section 6]{barles souganidis}).  

Recently, the author showed that these assumptions need not hold in general (e.g.\ even if $a$ and $W$ are smooth) and there are obstructions of a variational nature that make identifying the right class of coefficients a difficult problem \cite{pulsating einstein}.  Nonetheless, there is ample reason to believe that \cite{barles souganidis} is a step in the right direction and it is instructive to revisit those results.

The key assumptions made in \cite{barles souganidis} are as follows:
	\begin{itemize}
		\item[(i)] There is a family $\{U_{e}\}_{e \in S^{d-1}}$, smoothly parametrized by $e \in S^{d-1}$, of smooth solutions of the pulsating standing wave equation
			\begin{align} \label{E: pulsating standing wave}
				\mathcal{D}_{e}^{*} (a(y) \mathcal{D}_{e}U_{e}) + W'(U_{e}) &= 0 \quad \text{in} \, \, \mathbb{R} \times \mathbb{T}^{d}, \quad (\mathcal{D}_{e} := e \partial_{s} + D_{y}) \\ 
				\lim_{s \to \pm \infty} U_{e}(s,y) &= \pm 1, \quad \partial_{s} U_{e} \geq 0. \nonumber
			\end{align}
		\item[(ii)]  The linearized equation associated with \eqref{E: pulsating standing wave} is solvable.  The idea here is if $F_{e} \in L^{2}(\mathbb{R} \times \mathbb{T}^{d})$ and $e \in S^{d-1}$, then there should be a constant $\overline{F}_{e}$ and a corrector $P_{e}$ solving the degenerate elliptic cell problem:
			\begin{equation} \label{E: linearized equation}
				\mathcal{D}_{e}^{*}(a(y) \mathcal{D}_{e}P_{e}) + W''(U_{e}) P_{e} = F_{e} - \overline{F}_{e} \quad \text{in} \, \, \mathbb{R} \times \mathbb{T}^{d}.
			\end{equation}
		Since $\partial_{s} U_{e}$ is a positive eigenfunction of this operator, one anticipates that $\overline{F}_{e} = \|\partial_{s} U_{e}\|_{L^{2}(\mathbb{R} \times \mathbb{T}^{d})}^{-2} \int_{\mathbb{R} \times \mathbb{T}^{d}} F_{e}(s,y) \partial_{s} U_{e}(s,y) \, dy \, ds$.  
		
		In addition, \cite{barles souganidis} assumes that $e \mapsto \overline{F}_{e}$ and $e \mapsto P_{e}$ are smooth.  This presents problems even in the present work and seems to be intractable in general (see Remark \ref{R: degenerate directions} below), but we will see that the loss of smoothness can be overcome in some situations.
	\end{itemize}

Examples of coefficients $(a,W)$ where assumption (i) fails can be found in \cite{pulsating einstein}.  In fact, (i) turns out to be a constraint on the behavior of the energy $\mathcal{F}^{a}$, and, so far, the only known example where it holds is when $a$ is constant.  This difficulty notwithstanding, the motivation for this work stems from the additional degree of freedom in the mobility $m$.  When $m$ is non-constant, there are two, independent contributors to the effective behavior in the sharp interface limit in \eqref{E: general gradient flow}: the energy landscape of $\mathcal{F}^{a}$ and the oscillation of $m$.  Even if the role of the energy is not well understood, that still leaves the question of the asymptotics of \eqref{E: AC mobility}, where $a$ is constant but $m$ is oscillatory.

We answer that question here, proving that homogenization occurs and the sharp interface limit is described by \eqref{E: velocity}.  At a formal level, this is immediate given the approach of \cite{barles souganidis}.  However, even in this setting, a number of significant problems arise when turning the formal asymptotics into a proof, and these require new ideas.  We hope this paper will shed light on the difficulties that need to be overcome to treat \eqref{E: general gradient flow} in general, although there are still substantial problems that would need to be addressed first.

\subsection{Main result}  Before stating the result, here are the assumptions on $m$: 
	\begin{align}
		m &\in C(\mathbb{T}^{d} \times \mathbb{R}^{d}), \label{A: m continuous} \\
		0 < \theta &:= \min \left\{m(y,v) \, \mid \, (y,v) \in \mathbb{T}^{d} \times \mathbb{R}^{d}\right\}, \label{A: m positive} \\
		\Theta &:= \max \left\{ m(y,v) \, \mid \, (y,v) \in \mathbb{T}^{d} \times \mathbb{R}^{d} \right\} < \infty, \label{A: m bounded} \\
		H_{1} &= \sup \left\{ \frac{|m(y,v) - m(y',v)|}{\|y - y'\|} \, \mid \, (y,v), (y,v') \in \mathbb{T}^{d} \times \mathbb{R}^{d}, \, \, y \neq y' \right\} < \infty. \label{A: m holder}
	\end{align}
	
Concerning the potential, we assume that $W \in C^{3}([-1,1])$ satisfies
	\begin{align}
		\{W=0\} = \{-1,1\}, \quad &\{W' = 0\} = \{-1,0,1\}, \label{A: zeros} \\
		(-1,0) \subseteq \{W' > 0\}, \quad & (0,1) \subseteq \{W' < 0\}, \label{A: sign of derivative} \\
		W''(1) \wedge W''(-1) > 0, \quad & W''(0) < 0. \label{A: nondegeneracy of W}
	\end{align} 
	
The main result of the paper is
	
\begin{theorem} \label{T: main} Assume that $m$ satisfies \eqref{A: m continuous}, \eqref{A: m positive}, \eqref{A: m bounded}, and \eqref{A: m holder} and $W$ satisfies \eqref{A: zeros}, \eqref{A: sign of derivative}, and \eqref{A: nondegeneracy of W}.  Let $\overline{m} : S^{d-1} \to [\theta,\Theta]$ denote the effective mobility defined in \eqref{E: effective mobility} below.
If $u_{0} \in UC(\mathbb{R}^{d}; [-1,1])$, $(u^{\epsilon})_{\epsilon > 0} \subseteq C(\mathbb{R}^{d} \times [0,\infty); [-1,1])$ are the viscosity solutions of \eqref{E: AC mobility}, and $\bar{u} : \mathbb{R}^{d} \times [0,\infty) \to \mathbb{R}$ is the unique viscosity solution of the effective level set PDE
	\begin{equation} \label{E: effective equation}
		\left\{ \begin{array}{r l}
		\overline{m}(\widehat{D\bar{u}}) \bar{u}_{t} - \text{tr} \left( \left( \text{Id} - \widehat{D\bar{u}} \otimes \widehat{D\bar{u}} \right) D^{2} \bar{u} \right) = 0 & \text{in} \, \, \mathbb{R}^{d} \times (0,\infty), \\
		\bar{u} = u_{0} & \text{on} \, \, \mathbb{R}^{d} \times \{0\},
		\end{array} \right.
	\end{equation}
then $u^{\epsilon} \to \pm 1$ locally uniformly in $\{\pm \bar{u} > 0\}$.
\end{theorem}  

To prove Theorem \ref{T: main}, the starting point is the ansatz of \cite{barles souganidis}.  Namely, we consider the following asymptotic expansion for $u^{\epsilon}$:
	\begin{equation} \label{E: main ansatz}
		u^{\epsilon}(x,t) = q \left( \frac{\bar{d}(x,t)}{\epsilon} \right) + \epsilon P_{Dd(x,t)} \left(\frac{\bar{d}(x,t)}{\epsilon}, \frac{x}{\epsilon}\right) + \dots
	\end{equation} 
Here $q$ is the one-dimensional standing wave solution of \eqref{E: AC mobility} and we are led to the following cell problem for the correctors $(P_{e})_{e \in S^{d-1}}$ and effective mobility $\overline{m}$:
	\begin{equation} \label{E: cell problem}
		m(y,\dot{q}(s)e) \dot{q}(s) + \mathcal{D}_{e}^{*} \mathcal{D}_{e} P_{e} + W''(q(s)) P_{e} = \overline{m}(e) \dot{q}(s) \quad \text{in} \, \, \mathbb{R} \times \mathbb{T}^{d}.
	\end{equation}
Here, as before, $\mathcal{D}_{e} = e \partial_{s} + D_{y}$ so this is a degenerate elliptic PDE.  

As we briefly discuss in Remark \ref{R: degenerate directions} below, it turns out that these cell problems are ill-posed in general when $e$ is a rational direction, that is, $e \in \mathbb{R} \mathbb{Z}^{d}$.  This is precisely the obstruction to assumption (ii) above.  In some cases, it is possible to find solutions when $e$ is an irrational direction, that is, $e \notin \mathbb{R} \mathbb{Z}^{d}$, but it is not obvious this is true in general.

Nonetheless, we show below that, as long as $e \notin \mathbb{R} \mathbb{Z}^{d}$, \emph{approximate correctors} do exist.  More precisely, if $m$ and $W$ satisfy the assumptions of Theorem \ref{T: main} and $\nu > 0$, then there is an $\overline{m}(e) > 0$ and a $\tilde{P}^{\nu}_{e} \in C^{2,\mu}(\mathbb{R} \times \mathbb{T}^{d})$ such that
	\begin{align}
		 \left| [m(y,\dot{q}(s)e) - \overline{m}(e)] \dot{q}(s) + \mathcal{D}_{e}^{*} \mathcal{D}_{e} \tilde{P}^{\nu}_{e} + W''(q(s)) \tilde{P}^{\nu}_{e} \right|  \leq \nu \dot{q}(s) \quad \text{in} \, \, \mathbb{R} \times \mathbb{T}^{d}. \label{E: approximate correctors}	
	\end{align}

We will see that this is a good enough replacement of assumption (ii), but it leads to two problems that still need to be addressed.  First, the sub- and supersolutions used in \cite{barles souganidis} to control solutions of \eqref{E: AC mobility} as $\epsilon \to 0^{+}$ were global constructions that required correctors in every direction $e$.  That presents a problem here since there are obstructions to approximate correctors in rational directions.  To circumvent this, we build on the graphical approach of \cite{pulsating einstein}, showing how to construct local sub- and supersolutions controlling the developing interface in a neighborhood of any point where the normal vector is irrational.

Next, we need to show that knowledge of the behavior of the macroscopic interface at points where the normal is irrational is enough to determine its behavior globally.  This is too much to ask in general (e.g.\ if the interface is a plane with a rational normal), but we are saved by the fact that the structure of \eqref{E: AC mobility} already gives us a good amount of information at locations where the interface is more-or-less flat.  Using viscosity theoretic arguments, we prove that an interface that is only known to satisfy \eqref{E: velocity} in irrational directions and that is well-behaved at points where it is flat actually moves by \eqref{E: velocity} in the viscosity sense.  Put another way, a solution $\bar{u}$ of \eqref{E: effective equation} ``in irrational directions" is actually a viscosity solution in the usual sense.

\subsection{Related literature}  Allen and Cahn \cite{allen cahn} made the connection between mean curvature flow and the equation bearing their name.  This attracted quite a lot of attention in the mathematical community.  The first rigorous proofs appeared in the work of De Mottoni and Schatzman \cite{de mottoni schatzman}, Bronsard and Kohn \cite{bronsard kohn}, and Chen \cite{chen} in the setting of smooth flows.  The development of a viscosity theory for level set PDEs led to proofs that mean curvature motion describes the asymptotics globally in time.  This was the work of Evans, Soner, and Souganidis \cite{ESS}, Barles, Soner, and Souganidis \cite{barles soner souganidis}, and Barles and Souganidis \cite{barles souganidis}.  Shortly thereafter, Ilmanen \cite{ilmanen} showed how to describe the global asymptotics using Brakke flow and geometric measure theory.

In addition to the Allen-Cahn equation, considerable energy was invested in the motion of interfaces in the stochastic Ising model with K\`{a}c interactions.  See the work of De Masi, Orlandi, Presutti, and Triolo \cite{crazy de masi presutti paper}, which described the macroscopic evolution up to the development of singularities, and Katsoulakis and Souganidis \cite{katsoulakis souganidis isotropic}, \cite{katsoulakis souganidis anisotropic} for the global result (also see \cite{katsoulakis souganidis glauber kawasaki}).  In that context, the mescoscopic scale is described by a non-local reaction diffusion equation with a mobility term in the same spirit as \eqref{E: AC mobility}.  This partly inspired the present work.

Another source of inspiration is the discussion of phase separation phenomena by Taylor and Cahn \cite{taylor cahn}.  There the authors propose equations like \eqref{E: AC mobility}, except the mobility depends only on the direction of the gradient.  As acknowledged already in \cite{taylor cahn}, it is not clear that such equations are well-posed.  At any rate, from a modeling perspective, given that the Allen-Cahn functional is very sensitive to the norm of the gradient, it seems natural to consider mobility coefficients that depend on the amplitude of the gradient in addition to the direction as in \eqref{E: AC mobility}.

Beyond \cite{barles souganidis}, homogenization results for geometric motions in the parabolic scaling regime have been hard to come by.  Results have been obtained when the initial interface is a graph evolving in a laminar medium, starting with Barles, Cesaroni, and Novaga \cite{barles cesaroni novaga} for a class of forced mean curvature flows and, more recently, \cite{pulsating einstein} for phase field equations like \eqref{E: general gradient flow}.  For plane-like initial data, Cesaroni, Novaga, and Valdinoci \cite{cesaroni novaga valdinoci} prove homogenization of solutions of mean curvature flow with periodic forcing, obtaining effective front speeds that are discontinuous with respect to the normal direction (see also the paper by Chen and Lou \cite{chen lou}).  In the companion paper \cite{level set PDE}, the author proves homogenization for a broad class of geometric motions, roughly corresponding to those for which planes are stationary solutions, in dimension two, and homogenization of sub- and supersolutions in higher dimensions when the equation is quasi-linear.  In the higher dimensional case, the effective coefficients are generically discontinuous at every rational direction.

The fundamental difficulty of the present work and \cite{level set PDE} stems from the fact that the notion of corrector is not tractable in rational directions.  This is morally equivalent to the hurdles faced in the homogenization of oscillating boundary value problems.  (A non-exhaustive list of references is: the papers by Barles, Da Lio, and Souganidis \cite{barles da lio souganidis}, Barles and Mironescu \cite{barles mironescu}, G\'{e}rard-Varet and Masmoudi \cite{homogenization and boundary layers}, Choi and Kim \cite{choi kim}, Feldman \cite{feldman}, and Feldman and Kim \cite{feldman kim}.)  In that context as well as the present one, the fact that there is no well-defined ergodic constant in the cell problem in a given rational direction is intimately linked to the non-unique ergodicity of the associated group of translations on the torus.  In addition to complicating the analysis considerably, this can ultimately lead to discontinuities in the effective coefficients.

Finally, the use of approximate correctors in this work was very much inspired by what is by now a standard tool in the homogenization arsenal, see the papers by Ishii \cite{ishii almost periodic}, Lions and Souganidis \cite{lions souganidis front}, and Caffarelli, Souganidis, and Wang \cite{caffarelli souganidis wang elliptic}, for example.

\subsection{Outline of the paper}  Section \ref{S: strategy} details the strategy of the proof and states Theorem \ref{T: irrational directions}, which is needed to overcome the difficulty encountered in rational directions.  Section \ref{S: macroscopic part} treats the proof of Theorem \ref{T: irrational directions} and Section \ref{S: approximate correctors}, the existence of approximate correctors.  Section \ref{S: irrational directions} proves that the large-scale behavior of interfaces has the right behavior when the normal direction is irrational.  Section \ref{S: rational directions} treats the case when the normal direction is rational and the curvature is small.  In Section \ref{S: zero normal}, the case when ``the normal vanishes" is analyzed.  

There are three appendices.  Appendix \ref{A: initialization} contains a lemma used in the proof of homogenization; Appendix \ref{A: comparison principle} proves that \eqref{E: AC mobility} is well posed; and Appendix \ref{A: correctors} includes regularity results needed in the construction of correctors as well as a few standard facts about the one-dimensional Allen-Cahn equation.

\subsection{Notation}  In $\mathbb{R}^{d}$, $\langle \cdot, \cdot \rangle: \mathbb{R}^{d} \times \mathbb{R}^{d} \to \mathbb{R}$ is the Euclidean inner product and $\|\cdot\| : \mathbb{R}^{d} \to [0,\infty)$, the associated norm.  Given $p \in \mathbb{R}^{d} \setminus \{0\}$, we write $\hat{p} = \|p\|^{-1} p$.  

$\mathcal{L}^{d}$ is the Lebesgue measure in $\mathbb{R}^{d}$.  $\mathcal{H}^{d-1}$ is the $(d-1)$-dimensional Hasudorff measure, normalized to coincide with surface area.

Given $a, b \in \mathbb{R}$, we write $a \vee b = \max\{a,b\}$ and $a \wedge b = \min \{a,b\}$.

For a set $A$, the characteristic function $\chi_{A}$ is defined by $\chi_{A}(x) = 1$ if $x \in A$ and $\chi_{A}(x) = 0$, otherwise.  Given a finite measure $\mu$, $\fint_{A} (\cdot) \, \mu(dx) = \mu(A)^{-1} \int_{A} (\cdot) \, \mu(dx)$.

\section{Strategy of Proof} \label{S: strategy}

\subsection{Solutions of \eqref{E: effective equation} in Irrational Directions}  As explained in the introduction, the approximate correctors used in the asymptotic analysis are only available in irrational directions.  This leads to the question of whether or not an interface known to move by \eqref{E: velocity} at points where its normal vector is irrational actually has to move that way globally.  The case of a rational plane translating at an arbitrary velocity shows this is untenable.  However, the form of \eqref{E: AC mobility} already rules this out.  By building appropriate sub- and supersolutions, it is possible to prove that the interface obtained in the limit is a solution of \eqref{E: effective equation} in the following sense:

\begin{definition} \label{D: irrational directions} Given an open set $U \subseteq \mathbb{R}^{d} \times (0,\infty)$, we say that a locally bounded, upper semi-continuous function $\tilde{u} : \mathbb{R}^{d} \times (0,\infty) \to \mathbb{R}$ is a \emph{subsolution of \eqref{E: effective equation} in irrational directions in $U$} if there is a $K > 0$ such that, given a smooth function $\varphi : \mathbb{R}^{d} \times (0,\infty) \to \mathbb{R}$ such that $\tilde{u} - \varphi$ has a strict local maximum at $(x_{0},t_{0}) \in U$, the following conditions are met:
	\begin{itemize}
		\item[(a)] If $D\varphi(x_{0},t_{0}) \in \mathbb{R}^{d} \setminus \mathbb{R} \mathbb{Z}^{d}$, then 
			\begin{equation*}
				\overline{m}(\widehat{D\varphi}(x_{0},t_{0})) \varphi_{t}(x_{0},t_{0}) - \text{tr} \left( \left( \text{Id} - \widehat{D\varphi}(x_{0},t_{0}) \otimes \widehat{D\varphi}(x_{0},t_{0}) \right) D^{2} \varphi(x_{0},t_{0}) \right) \leq 0.
			\end{equation*}
		\item[(b)] If $D\varphi(x_{0},t_{0}) \in \mathbb{R} \mathbb{Z}^{d} \setminus \{0\}$ and 
			\begin{equation*}
				\left\| \left(\text{Id} - \widehat{D\varphi}(x_{0},t_{0}) \otimes \widehat{D\varphi}(x_{0},t_{0})\right) D^{2} \varphi(x_{0},t_{0}) \right\| \leq \delta \|D\varphi(x_{0},t_{0})\|,
			\end{equation*}
		then 
			\begin{equation*}
				\varphi_{t}(x_{0},t_{0}) \leq K \delta \|D\varphi(x_{0},t_{0})\|.
			\end{equation*}
		\item[(c)] If $\|D\varphi(x_{0},t_{0})\| = \|D^{2}\varphi(x_{0},t_{0})\| = 0$, then
			\begin{equation*}
				\varphi_{t}(x_{0},t_{0}) \leq 0.
			\end{equation*}
	\end{itemize}
	
Similarly, a locally bounded, lower semi-continuous function $\tilde{v} : \mathbb{R}^{d} \times (0,\infty) \to \mathbb{R}$ is a \emph{supersolution of \eqref{E: effective equation} in irrational directions in $U$} if $-\tilde{v}$ is a subsolution in irrational directions.  A locally bounded, continuous function in $\mathbb{R}^{d} \times (0,\infty)$ is a \emph{solution of \eqref{E: effective equation} in irrational directions in $U$} if it is both a subsolution and a supersolution.  \end{definition}

Using approximate correctors, we show that the interface obtained from \eqref{E: AC mobility} in the limit $\epsilon \to 0^{+}$ satisfies (a).  (b) and (c) then follow directly from the structure of \eqref{E: AC mobility} without correctors.  It only remains to determine whether or not Definition \ref{D: irrational directions} is enough to identify the viscosity solution of \eqref{E: effective equation}.  That is the content of the next theorem:

\begin{theorem} \label{T: irrational directions} Given an open set $U \subseteq \mathbb{R}^{d} \times (0,\infty)$, if $\tilde{u} : \mathbb{R}^{d} \times (0,\infty) \to \mathbb{R}$ is a subsolution (resp.\ supersolution) of \eqref{E: effective mobility} in irrational directions in $U$, then it is a viscosity subsolution (resp.\ supersolution) of \eqref{E: effective mobility} in $U$ in the usual sense.  \end{theorem}  

In view of the theorem, to prove Theorem \ref{T: main} it suffices to show that the interface obtained from \eqref{E: AC mobility} is a solution of \eqref{E: effective equation} in irrational directions.  The rest of this section lays out the main steps of the argument.

\subsection{Steps of the Proof}  Henceforth $u_{0} \in UC(\mathbb{R}^{d}; [-1,1])$ is fixed and $(u^{\epsilon})_{\epsilon > 0}$ are the solutions of \eqref{E: AC mobility} in $\mathbb{R}^{d} \times (0,\infty)$.  The existence and uniqueness of these solutions is sketched in Appendix \ref{A: comparison principle}.

The macroscopic phases that develop in the sharp interface limit are described by the following open sets, parametrized by $t \geq 0$:
	\begin{align*}
		\Omega_{t}^{(1)} &= \left\{ x \in \mathbb{R}^{d} \, \mid \, \liminf \nolimits_{*} u^{\epsilon}(x,t) = 1 \right\}, \\
		\Omega_{t}^{(2)} &= \left\{ x \in \mathbb{R}^{d} \, \mid \, \limsup \nolimits^{*} u^{\epsilon}(x,t) = -1 \right\}.
	\end{align*}
Recall that the half-relaxed limits in the definition are given by
	\begin{align*}
		\liminf \nolimits_{*} u^{\epsilon}(x,t) &= \lim_{\delta \to 0^{+}} \inf \left\{ u^{\epsilon}(y,s) \, \mid \, \epsilon + \|x - y\| + |t - s| \leq \delta, \, \, s > 0 \right\}, \\
		\limsup \nolimits^{*} u^{\epsilon}(x,t) &= \lim_{\delta \to 0^{+}} \sup \left\{ u^{\epsilon}(y,s) \, \mid \, \epsilon + \|x - y\| + |t - s| \leq \delta, \, \, s > 0 \right\}.
	\end{align*}
	
We proceed by proving that $(\Omega_{t}^{(1)})_{t \geq 0}$ and $(\Omega_{t}^{(2)})_{t \geq 0}$ define super- and subflows of \eqref{E: effective mobility}, respectively, in the sense of \cite{barles souganidis}.  It will not be necessary to know what that means in this paper.  Instead, we define $\chi^{*}, \chi_{*} : \mathbb{R}^{d} \times [0,\infty) \to \mathbb{R}$ by 
	\begin{align*}
		\chi_{*}(x,t) = \left\{ \begin{array}{r l}
							1, & \text{if} \, \, x \in \Omega_{t}^{(1)}, \\
							-1, & \text{otherwise},
						\end{array} \right. 
		\quad \chi^{*}(x,t) = \left\{ \begin{array}{r l}
								1, & \text{if} \, \, x \in \overline{\Omega_{t}^{(2)}}, \\
								-1, & \text{otherwise},
					\end{array} \right.
	\end{align*}
and prove these are respectively discontinuous super- and subsolutions of \eqref{E: effective equation}.
The main thrust of the paper is the proof of the following result:
		
	\begin{prop} \label{P: main proposition} $\chi^{*}$ (resp.\ $\chi_{*}$) is a subsolution (resp.\ supersolution) of \eqref{E: effective mobility} in irrational directions in $\mathbb{R}^{d} \times (0,\infty)$.  Furthermore, $\chi^{*}(\cdot,0) = -1$ in $\{u_{0} < 0\}$ and $\chi_{*}(\cdot,0) = 1$ in $\{u_{0} > 0\}$. \end{prop}  
	
	\begin{proof}  Where $\chi_{*}$ is concerned, the first statement follows from Propositions \ref{P: graph localization}, \ref{P: rational directions}, and \ref{P: zero normal}, and the second, from Proposition \ref{P: initial data} below.  The statements concerning $\chi^{*}$ then follow by replacing $u^{\epsilon}$ by $-u^{\epsilon}$, $W$ by $u \mapsto W(-u)$, $\chi^{*}$ by $-\chi^{*}$, etc.\ since this has the effect of transforming supersolutions into subsolutions.  \end{proof} 
	
In view of Theorem \ref{T: irrational directions}, we can remove the ``irrational directions" qualifier and instead treat $\chi^{*}$ and $\chi_{*}$ as viscosity sub- and supersolutions.  It only remains to prove that, even though these functions are discontinuous, $\chi^{*}$ and $\chi_{*}$ can still be compared to $\bar{u}$.  This part is classical.

	\begin{theorem} \label{T: comparison principle}  If $\bar{u}$ is the solution of \eqref{E: effective mobility}, then $\chi^{*} \leq - \chi_{\{\bar{u} < 0\}}$ and $\chi_{*} \geq \chi_{\{\bar{u} > 0\}}$ in $\mathbb{R}^{d} \times [0,\infty)$.  \end{theorem}  
	
		\begin{proof}  The existence and uniqueness of $\bar{u}$ is standard (cf.\ \cite{barles soner souganidis}).  Since $u_{0}$ is uniformly continuous and the equation is translationally invariant, a well-known (e.g.\ approximation) argument shows that $\bar{u}$ is uniformly continuous in both variables.  Thus, \cite[Proposition 2.1]{barles souganidis} applies, giving the desired conclusion.  \end{proof}  
	
Now notice that Theorem \ref{T: comparison principle} implies Theorem \ref{T: main} by definition of $\chi^{*}$ and $\chi_{*}$.  Therefore, it only remains to prove Proposition \ref{P: main proposition} and Theorem \ref{T: irrational directions}.

\section{Proof of Theorem \ref{T: irrational directions}}  \label{S: macroscopic part}

\subsection{Sketch of proof}  In this section, we prove Theorem \ref{T: irrational directions}.  Here it is important to recall a very convenient (equivalent) notion of solution of \eqref{E: effective equation} that goes back to Barles and Georgelin \cite{barles georgelin}.  

	\begin{definition}  \label{D: barles georgelin} Given an open set $U \subseteq \mathbb{R}^{d} \times (0,\infty)$, a locally bounded, upper semi-continuous function $\tilde{u} : \mathbb{R}^{d} \times (0,\infty) \to \mathbb{R}$ is a \emph{viscosity subsolution of \eqref{E: effective equation} in $U$} provided it has the following property: given a smooth function $\varphi : \mathbb{R}^{d} \times (0,\infty) \to \mathbb{R}$ is a smooth function and $(x_{0},t_{0}) \in U$ such that strict local maximum of $\tilde{u} - \varphi$ has a strict local maxium at $(x_{0},t_{0})$, 
		\begin{itemize}
			\item[(a)] If $D\varphi(x_{0},t_{0}) \neq 0$, then
				\begin{equation*}
					\overline{m}(\widehat{D\varphi}(x_{0},t_{0})) \varphi_{t}(x_{0},t_{0}) - \text{tr} \left( \left( \text{Id} - \widehat{D\varphi}(x_{0},t_{0}) \otimes \widehat{D\varphi}(x_{0},t_{0}) \right) D^{2} \varphi(x_{0},t_{0}) \right) \leq 0.
				\end{equation*}
			\item[(b)] If $\|D\varphi(x_{0},t_{0})\| = \|D^{2} \varphi(x_{0},t_{0})\| = 0$, then
				\begin{equation*}
					\varphi_{t}(x_{0},t_{0}) \leq 0.
				\end{equation*}
		\end{itemize}
	
Similarly, a locally bounded, lower semi-continuous function $\tilde{v} : \mathbb{R}^{d} \times (0,\infty) \to \mathbb{R}$ is a \emph{viscosity supersolution of \eqref{E: effective equation} in $U$} if $-\tilde{v}$ is a viscosity subsolution.  A locally bounded, continuous function in $\mathbb{R}^{d} \times (0,\infty)$ is a \emph{viscosity solution of \eqref{E: effective equation} in $U$} if it is both a subsolution and a supersolution.
\end{definition}  

With Definition \ref{D: barles georgelin} in mind, formally, the reason Theorem \ref{T: main} is true is if the solution in question were smooth, then it would be clear.  Indeed, we only need to check points where $Du(x_{0},t_{0}) \neq 0$ since otherwise Definition \ref{D: barles georgelin}, (b) and Definition \ref{D: irrational directions}, (iii) are in agreement.  If $Du(x_{0},t_{0}) \in \mathbb{R} \mathbb{Z}^{d}$, then either $Du(x,t) \notin \mathbb{R} \mathbb{Z}^{d}$ for some $(x,t)$ arbitrarily close to $(x_{0},t_{0})$, in which case the necessary differential inequality follows by continuity, or $Du(x,t) = Du(x_{0},t_{0})$ in a neighborhood of $(x_{0},t_{0})$.  In the latter case, differentiation shows
	\begin{equation*}
		\left(\text{Id} - \widehat{Du}(x_{0},t_{0}) \otimes \widehat{Du}(x_{0},t_{0})\right) D^{2} u(x_{0},t_{0}) = 0
	\end{equation*}
and now we are in the purview of Definition \ref{D: irrational directions}, (ii).

The sub- and supersolutions we work with in the proof are discontinuous, being characteristic functions of open sets, and, thus, far from smooth.  To circumvent this, we show that the sketch above is correct when $u$ is semi-convex or semi-concave and then use sup- and inf-convolutions to pass to the general case.  

\subsection{Proof of Theorem \ref{T: main}}  In order to make the previous sketch rigorous in the case of a semi-convex/semi-concave function, we will need to use facts about the derivatives of such functions.  

First, we recall that a semi-convex/semi-concave function has a derivative in $BV_{\text{loc}}$:

	\begin{lemma} \label{L: semi-convex} If $\Omega \subseteq \mathbb{R}^{d}$ is a bounded open set and $u : \Omega \to \mathbb{R}$ is semi-convex or semi-concave, then $Du \in BV_{\text{loc}}(\Omega;\mathbb{R}^{d})$ and the absolutely continuous part of the derivative of $Du$ coincides with $D^{2}u$ $\mathcal{L}^{d}$-a.e.\ in $\Omega$.
	\end{lemma}  
		
Next, we show that the differentiation step in the sketch can be made rigorous even in the semi-convex/semi-concave case.  In the lemma below, we have in mind that $V$ is the derivative of a semi-convex/semi-concave function.

	\begin{prop} \label{P: fine}  Suppose $\Omega \subseteq \mathbb{R}^{d}$ is a bounded open set and $V \in BV_{\text{loc}}(\Omega; \mathbb{R}^{m})$ for some $m \in \mathbb{N}$.  Let $D^{\text{ac}}V \in L^{1}_{\text{loc}}(\Omega; \mathbb{R}^{d \times m})$ denote the Radon-Nikodym derivative of the Radon measure $DV$ with respect to $\mathcal{L}^{d}$.  Given any $v \in \mathbb{R}^{d}$ and $e \in S^{d-1}$, 
		\begin{equation*}
			D^{\text{ac}}V = 0 \quad \mathcal{L}^{d}\text{-a.e.\ in} \, \, \{V = v\}
		\end{equation*}    
	and 
		\begin{equation*}
			\left(\text{Id} - \widehat{V} \otimes \widehat{V} \right) D^{\text{ac}}V = 0 \quad \mathcal{L}^{d}\text{-a.e.\ in} \, \, \{\widehat{V} = e\}
		\end{equation*}    
	\end{prop}  
	
		\begin{proof}  Let $\mathcal{D}_{V}$ denote the set of approximate differentiability points of $V$, that is, $x \in \mathcal{D}_{V}$ if and only if there is a linear map $A_{x} : \mathbb{R}^{d} \to \mathbb{R}^{m}$ such that
			\begin{equation*}
				\lim_{r \to 0^{+}} r^{-d} \int_{B(x,r)} \frac{\|V(y) - V(x) - A_{x}(y -x)\|}{r} \, dy = 0.
			\end{equation*} 
		Since $V \in BV_{\text{loc}}(\Omega; \mathbb{R}^{m})$, it follows that $\mathcal{L}^{d}(\Omega \setminus \mathcal{D}_{V}) = 0$ and $A_{x} = D^{\text{ac}}V(x)$ for a.e.\ $x \in \Omega$ (cf.\ \cite[Theorem 3.83]{ambrosio fusco pallara}).  
		
		A straightforward computation shows $D^{\text{ac}}V = 0$ a.e.\ in $\{V = v\}$ (see also \cite[Proposition 3.73]{ambrosio fusco pallara}).  
		
		Define $f : \mathbb{R}^{d} \to \mathbb{R}^{d}$ by $f(p) = \widehat{p}$ if $p \neq 0$ and $f(0) = 0$.  It is not hard to see that each $x \in \mathcal{D}_{V}$ with $A_{x} \neq 0$ is an approximate differentiability point of $f(V)$.  Furthermore, a straightforward computation shows its approximate derivative is given by $\|V\|^{-1} \left(\text{Id} - \widehat{V} \otimes \widehat{V}\right) D^{\text{ac}}V$ a.e.\  (Both statements can be found in \cite[Proposition 3.71]{ambrosio fusco pallara}).  As in the case of $\{V = v\}$, it is not hard to see that the approximate derivative of $f(V)$ has to vanish a.e.\ in $\{f(V) = e\} = \{\widehat{V} = e\}$.       \end{proof}  
		
Now we are prepared for the

	\begin{proof}[Proof of Theorem \ref{T: irrational directions}]  Let $V$ be a bounded, open subset of $U$ such that $\overline{V}$ is compactly contained in $U$.  For convenience, pick a $T > 0$ such that $V \subseteq \mathbb{R}^{d} \times (0,T)$.
	
	For each $\eta > 0$, define the inf-convolution $\tilde{u}_{\eta}$ of $\tilde{u}$ by
		\begin{equation*}
			\tilde{u}_{\eta}(x,t) = \inf \left\{ \tilde{u}(y,s) + \frac{\|y - x\|^{2}}{2 \eta} + \frac{(t - s)^{2}}{2 \eta} \, \mid \, (y,s) \in U \cap (\mathbb{R}^{d} \times (0,T)) \right\}.
		\end{equation*}
	A classical argument shows that $\tilde{u}_{\eta}$ is a semi-concave function in $U \cap (\mathbb{R}^{d} \times (0,T))$.
	
	Another well-known argument shows that there is an $\eta_{0} > 0$ such that if $0 < \eta < \eta_{0}$, then, for each $(x,t) \in V$, there is a $(y,s) \in U \cap (\mathbb{R}^{d} \times (0,T))$ such that 
		\begin{equation*}
			\tilde{u}_{\eta}(x,t) = \tilde{u}(y,s) + \frac{\|y - x\|^{2}}{2 \eta} + \frac{(t - s)^{2}}{2 \eta}.
		\end{equation*}
	From this, we can argue as in \cite{user's guide} to show that $\tilde{u}_{\eta}$ is a subsolution of \eqref{E: effective mobility} in irrational directions in $V$.  In fact, we claim that, for each $\eta \in (0,\eta_{0})$, $\tilde{u}_{\eta}$ satisfies
		\begin{equation} \label{E: keystone}
			\overline{m}(\widehat{D\tilde{u}^{\eta}}) \tilde{u}_{\eta,t} - \text{tr}_{*} \left( \left(\text{Id} - \widehat{D\tilde{u}^{\eta}} \otimes \widehat{D\tilde{u}^{\eta}} \right) D^{2}\tilde{u}^{\eta} \right) \geq 0 \quad \text{in} \, \, V.
		\end{equation}
	
	Henceforth, fix $\eta \in (0,\eta_{0})$ and let us proceed to the proof of \eqref{E: keystone}.  Assume that $\varphi$ is smooth and $\tilde{u}_{\eta} - \varphi$ has a strict local minimum at $(x_{0},t_{0}) \in V$.  If $\|D\varphi(x_{0},t_{0})\| = \|D^{2}\varphi(x_{0},t_{0})\| = 0$ or $\widehat{D\varphi}(x_{0}) \notin \mathbb{R} \mathbb{Z}^{d}$, then there is nothing to show since $\tilde{u}_{\eta}$ is a supersolution in irrational directions in $V$.  Thus, it remains to consider the case when $D\varphi(x_{0}) \neq 0$ and $\widehat{D\varphi}(x_{0}) \in \mathbb{R} \mathbb{Z}^{d}$.  
	
	In what follows (subtracting a constant from $\varphi$ if necessary), let $s > 0$ be such that $B((x_{0},t_{0}),s) \subseteq V$, $\tilde{u}_{\eta}(x,t) > \varphi(x)$ for $x \in B((x_{0},t_{0}),s) \setminus \{(x_{0},t_{0})\}$, and $\tilde{u}_{\eta}(x_{0},t_{0}) = \varphi(x_{0},t_{0})$.
	
	Since $\tilde{u}_{\eta}$ is semi-concave and $\varphi$ is smooth, it follows that $\tilde{u}_{\eta} - \varphi$ is semi-concave in $B((x_{0},t_{0}),s)$.  Thus, for each $\delta > 0$ small enough, we can apply Jensen's Lemma (cf.\ \cite{user's guide}), thereby obtaining a set $K_{\delta} \subseteq B((x_{0},t_{0}),s)$ such that $\mathcal{L}^{d + 1}(K_{\delta}) > 0$ and, for each $(x,t) \in K_{\delta}$,
	\begin{itemize}
		\item[(i)] $\tilde{u}_{\eta}$ is twice punctually differentiable at $(x,t)$.
		\item[(ii)] There is an $(a_{(x,t)},p_{(x,t)}) \in B(0,\delta)$ such that the function $(y,r) \mapsto \tilde{u}_{\eta}(y,r) - \varphi(y,r) - \langle p_{(x,t)}, y \rangle - a_{(x,t)} r$ has a local minimum in $B((x_{0},t_{0}),s)$.
	\end{itemize}
	We claim that we can find an $(x_{1},t_{1}) \in K_{\delta}$ such that
		\begin{align} 
			&-\delta\Theta \leq \overline{m} \left(\frac{D\varphi(x_{1},t_{1}) + p_{(x_{1},t_{1})}}{\|D\varphi(x_{1},t_{1}) + p_{(x_{1},t_{1})}\|}\right) \varphi_{t}(x_{1},t_{1}) \label{E: important goal} \\
				&\quad - \text{tr} \left( \left( \text{Id} - \frac{[D\varphi(x_{1},t_{1}) + p_{(x_{1},t_{1})}] \otimes [D\varphi(x_{1},t_{1}) + p_{(x_{1},t_{1})}]}{\|D\varphi(x_{1},t_{1}) + p_{(x_{1},t_{1})}\|^{2}} \right) D^{2} \varphi(x_{1},t_{1}) \right) \nonumber.
		\end{align}
	Making $\delta$ and $s$ smaller if necessary, we can assume that 
		\begin{equation*}
			D\tilde{u}_{\eta}(x_{1},t_{1}) = D\varphi(x_{1},t_{1}) + p_{(x_{1},t_{1})} \neq 0 \quad  \text{for all}  \quad (x_{1},t_{1}) \in K_{\delta}.
		\end{equation*} 

We will prove \eqref{E: important goal} by studying the structure of the spatial derivatives of $\tilde{u}_{\eta}$ in $K_{\delta}$.  We only need to consider the following two cases:
	\begin{itemize}
		\item[(a)] For a.e.\ $(x,t) \in K_{\delta}$, $D\tilde{u}_{\eta}(x,t) \in \mathbb{R} \mathbb{Z}^{d} \setminus \{0\}$.
		\item[(b)] There is a measurable $A_{\delta} \subseteq K_{\delta}$ such that $D\tilde{u}_{\eta}(x,t) \in \mathbb{R}^{d} \setminus \mathbb{R} \mathbb{Z}^{d}$ for a.e.\ $(x,t) \in A_{\delta}$ and $\mathcal{L}^{d + 1}(A_{\delta}) > 0$.
		\end{itemize}

The easier case is (b).  If (b) holds, then we can fix an $(x_{1},t_{1}) \in A_{\delta}$ and invoke the supersolution property of $\tilde{u}_{\eta}$ at $(x_{1},t_{1})$ to find
	\begin{equation*}
		0 \leq \overline{m}(\widehat{D\tilde{u}_{\eta}}(x_{1},t_{1})) \tilde{u}_{\eta,t}(x_{1},t_{1}) - \text{tr} \left( \left(\text{Id} - \widehat{D\tilde{u}_{\eta}}(x_{1},t_{1}) \otimes \widehat{D\tilde{u}_{\eta}}(x_{1},t_{1}) \right) D^{2} \tilde{u}_{\eta}(x_{1},t_{1}) \right).
	\end{equation*}
Now we recall that, by (ii), the following relations hold:
	\begin{align*}
		D\tilde{u}_{\eta}(x_{1},t_{1}) &= D\varphi(x_{1},t_{1}) + p_{(x_{1},t_{1})}, \quad D^{2}\tilde{u}_{\eta}(x_{1},t_{1}) \geq D^{2}\varphi(x_{1},t_{1}), \\
		\tilde{u}_{\eta,t}(x_{1},t_{1}) &= \varphi_{t}(x_{1},t_{1}) + a_{(x_{1},t_{1})}.
	\end{align*}
By ellipticity and \eqref{A: m positive}, this gives \eqref{E: important goal}.  
	
Next, we turn to case (a).  Given $t \in (0,T)$, let $U_{t} = \{x \in \mathbb{R}^{d} \, \mid \, (x,t) \in U\}$.  Recall from Proposition \ref{P: fine} that the map $D\tilde{u}_{\eta}(\cdot,t) \in BV_{\text{loc}}(U_{t}; \mathbb{R}^{d})$ for each fixed $t$ and $D^{\text{ac}}(D\tilde{u}_{\eta}(\cdot,t)) = D^{2}\tilde{u}_{\eta}(\cdot,t)$ a.e.  Let us define $\{B_{e}\}_{e \in S^{d-1} \cap \mathbb{R} \mathbb{Z}^{d}}$ by  
	\begin{align*}
		B_{e} &= \left\{(x,t) \in K_{\delta} \, \mid \, \frac{D\tilde{u}_{\eta}(x,t)}{\|D\tilde{u}_{\eta}(x,t)\|} = e \right\}.
	\end{align*}
Since we assumed (a) holds, it follows that $\sum_{e} \mathcal{L}^{d + 1}(B_{e}) > 0$.

An immediate application of Lemma \ref{L: semi-convex}, Proposition \ref{P: fine}, and Fubini's Theorem shows that 
	\begin{align*}
		\left(\text{Id} - \widehat{D\tilde{u}_{\eta}} \otimes \widehat{D\tilde{u}_{\eta}}\right) D^{2} \tilde{u}_{\eta} &= 0 \quad \text{a.e.\ in} \, \, \bigcup_{e \in S^{d-1} \cap \mathbb{R} \mathbb{Z}^{d}} B_{e}.
	\end{align*}  
Thus, we can fix a point $(x_{1},t_{1}) \in \bigcup_{e} B_{e}$ such that 
	\begin{equation*}
			\left(\text{Id} - \widehat{D\tilde{u}_{\eta}}(x_{1},t_{1}) \otimes \widehat{D\tilde{u}_{\eta}}(x_{1},t_{1})\right) D^{2}\tilde{u}_{\eta}(x_{1},t_{1}) = 0.
	\end{equation*}  
Since $\bigcup_{e \in S^{d-1} \cap \mathbb{R} \mathbb{Z}^{d}} B_{e} \subseteq K_{\delta}$ and $\tilde{u}_{\eta}$ is a supersolution of \eqref{E: effective equation} in irrational directions in $V$, we have
	\begin{align*}
		0 &\leq \overline{m} \left(\widehat{D\tilde{u}_{\eta}}(x_{1},t_{1}) \right) \tilde{u}_{\eta,t}(x_{1},t_{1})  \\
			&\leq \overline{m} \left( \frac{D\varphi(x_{1},t_{1}) + p_{(x_{1},t_{1})}}{\|D\varphi(x_{1},t_{1}) + p_{(x_{1},t_{1})}\|} \right) \varphi_{t}(x_{1},t_{1}) + \delta \Theta \\
				&\quad - \text{tr} \left( \left(\text{Id} - \frac{[D\varphi(x_{1},t_{1}) + p_{(x_{1},t_{1})}] \otimes [D\varphi(x_{1},t_{1}) + p_{(x_{1},t_{1})}]}{\|D\varphi(x_{1},t_{1}) + p_{(x_{1},t_{1})}\|^{2}} \right) D^{2}\varphi(x_{1},t_{1}) \right).
	\end{align*}
This is exactly \eqref{E: important goal}.  
		
We conclude that, no matter which of cases (a) or (b) occur, there is an $(x_{1},t_{1}) \in K_{\delta}$ such that \eqref{E: important goal} holds.  Next, by recalling that $(x_{0},t_{0})$ is a strict local minimum of $\tilde{u}_{\eta} -\varphi$ in $B((x_{0},t_{0}),s)$ and $K_{\delta} \subseteq B((x_{0},t_{0}),s)$, a straightforward argument shows that $(x_{1},t_{1}) \to (x_{0},t_{0})$ as $\delta \to 0^{+}$.  Therefore, in the limit $\delta \to 0^{+}$, \eqref{E: important goal} becomes
	\begin{equation*}
		0 \leq \overline{m}(\widehat{D\varphi}(x_{0},t_{0})) \varphi_{t}(x_{0},t_{0}) - \text{tr} \left( \left( \text{Id} - \widehat{D\varphi}(x_{0},t_{0}) \otimes \widehat{D\varphi}(x_{0},t_{0}) \right) D^{2} \varphi(x_{0},t_{0}) \right).
	\end{equation*}
	
Since $\varphi$ was arbitrary, we proved that \eqref{E: keystone} holds as long as $\eta \in (0,\eta_{0})$.  At the same time, we know that $\tilde{u} = \liminf_{*} \tilde{u}_{\eta}$ in $V$ as $\eta \to 0^{+}$.  Thus, the stability properties of viscosity solutions imply that $\tilde{u}$ is also a supersolution, that is,
	\begin{equation*}
		\overline{m}(\widehat{D\tilde{u}}) \tilde{u}_{t} - \text{tr}_{*} \left( \left(\text{Id} - \widehat{D\tilde{u}} \otimes \widehat{D\tilde{u}} \right) D^{2} \tilde{u} \right) \geq 0 \quad \text{in} \, \, V.
	\end{equation*}
Since $V$ was arbitrary, we conclude that $\tilde{u}$ is a viscosity supersolution of \eqref{E: effective mobility} in $U$.\end{proof}

\section{Approximate Correctors} \label{S: approximate correctors}

The purpose of this section is to prove the existence of approximate correctors, that is, solutions of \eqref{E: approximate correctors}.  When $e \in S^{d-1} \setminus \mathbb{R} \mathbb{Z}^{d}$, this is possible because the diffusion in the $\langle e \rangle^{\perp}$ directions explores the entire torus.  We will show below that the same strategy does not work in rational directions precisely because this is no longer the case.  In Remark \ref{R: cell problem existence}, we show that when $m$ is sufficiently regular, \eqref{E: cell problem} has solutions in certain Diophantine directions; Remark \ref{R: degenerate directions} shows that there is an obstruction when $e \in \mathbb{R} \mathbb{Z}^{d}$.    

We will see below that the effective mobility $\overline{m}(e)$ is given by 
	\begin{equation} \label{E: effective mobility}
		\overline{m}(e) = c_{W}^{-1} \int_{\mathbb{R} \times \mathbb{T}^{d}} m(y,\dot{q}(s)e) \dot{q}(s)^{2} \, dy \, ds, \quad c_{W} = \int_{-\infty}^{\infty} \dot{q}(s)^{2} \, ds.
	\end{equation}

The main result of this section concerning existence of approximate correctors is stated next:

	\begin{theorem} \label{T: approximate correctors}  Fix $e \in S^{d-1} \setminus \mathbb{R} \mathbb{T}^{d}$.  If $m$ and $W$ satisfy the assumptions of Theorem \ref{T: main} and $\nu > 0$, then there is a $\tilde{P}^{\nu}_{e} \in C^{2,\mu}(\mathbb{R} \times \mathbb{T}^{d})$ such that \eqref{E: approximate correctors} holds. 
	\end{theorem}
	
To prove the theorem, we start by regularizing $m$: given $\mu \in (0,1)$, fix $\tilde{m} \in C^{\mu}(\mathbb{R} \times \mathbb{T}^{d})$ such that
	\begin{align}
		\sup \left\{ |m(y,\dot{q}(s)e) - \tilde{m}(s,y)| \, \mid \, (s,y) \in \mathbb{R} \times \mathbb{T}^{d} \right\} \leq \frac{1}{3} \nu, \label{E: close to m} \\
		\|\tilde{m}\|_{C^{\mu}(\mathbb{R} \times \mathbb{T}^{d})} + \|D_{y} \tilde{m}\|_{C^{\mu}(\mathbb{R} \times \mathbb{T}^{d})} + \|D^{2}_{y} \tilde{m}\|_{C^{\mu}(\mathbb{R} \times \mathbb{T}^{d})} < \infty. \label{E: smooth}
	\end{align}
We decompose $\tilde{m}$ in the following way:
	\begin{equation}
		\tilde{m}(s,y) = \int_{\mathbb{T}^{d}} \tilde{m}(s,y') \, dy' + \left(\tilde{m}(s,y) - \int_{\mathbb{T}^{d}} \tilde{m}(s,y') \, dy' \right) =: \tilde{m}_{1}(s) + \tilde{m}_{2}(s,y).
	\end{equation}
Correspondingly, we define a corrector $\overline{P}_{e}$ and penalized correctors $(P^{\delta}_{2})_{\delta > 0}$ solving the following PDE:
	\begin{align}
		\tilde{m}_{1}(s) \dot{q}(s) - \ddot{\overline{P}}_{e} + W''(q(s)) \overline{P}_{e} &= \overline{\tilde{m}}(e) \dot{q}(s) \quad \text{in} \, \, \mathbb{R}, \\
		\tilde{m}_{2}(s,y) \dot{q}(s) + \delta P_{2}^{\delta} + \mathcal{D}_{e}^{*} \mathcal{D}_{e} P_{2}^{\delta} + W''(q(s)) P_{2}^{\delta} &= 0 \quad \text{in} \, \, \mathbb{R} \times \mathbb{T}^{d}, \label{E: fluctuation cell problem} \\
		\overline{\tilde{m}}(e) = c_{W}^{-1} \int_{\mathbb{R} \times \mathbb{T}^{d}} \tilde{m}(s,y) \dot{q}(s)^{2} \, dy \, ds.&
	\end{align}
Here, as above, $\mathcal{D}_{e} = e \partial_{s} + D_{y}$.  

The existence and regularity of $\overline{P}_{e}$ and $(P^{\delta}_{e})_{\delta > 0}$ is discussed in Appendix \ref{A: correctors}.  Theorem \ref{T: approximate correctors} is proved as soon as we establish the following:

	\begin{prop} \label{P: convergence of fluctuations}  $\|\dot{q}^{-1} (\delta P^{\delta}_{2})\|_{L^{\infty}(\mathbb{R} \times \mathbb{T}^{d})} \to 0$ as $\delta \to 0^{+}$.  In particular, given $\nu > 0$, if $\delta > 0$ is small enough, then $\tilde{P}^{\nu}_{e} = \overline{P}_{e} + P^{\delta}_{e}$ satisfies \eqref{E: approximate correctors}.  \end{prop}

\subsection{Convergence of $\delta P^{\delta}_{2}$}  Since $(P^{\delta}_{2})_{\delta > 0} \subseteq C^{2,\alpha}(\mathbb{R} \times \mathbb{T}^{d})$, straightforward manipulations show that the functions $(v^{\delta}_{\zeta})_{\zeta \in \mathbb{R}}$ obtained by the rule
	\begin{equation*}
		v^{\delta}_{\zeta}(x) = V^{\delta}_{2}(\langle x,e \rangle - \zeta, x), \quad V_{2}^{\delta}(s,y) = \dot{q}(s)^{-1} P_{2}^{\delta}(s,y)
	\end{equation*}
are solutions of the following family of PDE:
	\begin{equation*}
		\tilde{m}_{2}(\langle x,e \rangle - \zeta,x) + \delta v^{\delta}_{\zeta} - \Delta v^{\delta}_{\zeta} - \frac{2 \ddot{q}(\langle x,e \rangle - \zeta)}{\dot{q}(\langle x,e \rangle - \zeta)} \langle e, Dv^{\delta}_{\zeta} \rangle = 0 \quad \text{in} \, \, \mathbb{R}^{d}.
	\end{equation*}
Thus, the asymptotic behavior of $(\delta P^{\delta}_{2})_{\delta > 0}$ is captured by that of $(\delta v^{\delta}_{\zeta})_{\delta > 0}$.  

Notice that if we define $(\tilde{v}^{\delta}_{\zeta})_{\zeta \in \mathbb{R}}$ by 
	\begin{equation} \label{E: recentered solutions}
		\tilde{v}^{\delta}_{\zeta}(x) = v^{\delta}_{\zeta}(x + \zeta e),
	\end{equation}
then the functions $(\tilde{v}^{\delta}_{\zeta})_{\zeta \in \mathbb{R}}$ satisfy the ``centered" family of PDE:
	\begin{equation} \label{E: recentered equations}
		\tilde{m}_{2}(\langle x,e \rangle, x + \zeta e) + \delta \tilde{v}^{\delta}_{\zeta} - \Delta \tilde{v}^{\delta}_{\zeta} - \frac{2 \ddot{q}(\langle x,e \rangle)}{\dot{q}(\langle x,e \rangle)} \langle e, D\tilde{v}_{\zeta}^{\delta} \rangle = 0 \quad \text{in} \, \, \mathbb{R}^{d}.
	\end{equation}
Therefore, letting $\tilde{p} : \mathbb{R} \times \mathbb{R} \times (0,\infty) \to (0,\infty)$ and $g : \langle e \rangle^{\perp} \times \langle e \rangle^{\perp} \times (0,\infty) \to (0,\infty)$ denote the fundamental solutions of the operators $\partial_{t} - \partial_{ss} + W''(q(s))$ and $\partial_{t} - \Delta_{\langle e \rangle^{\perp}}$, respectively, we find
	\begin{align}
		\tilde{v}^{\delta}_{\zeta}(x) &= -\int_{0}^{\infty} e^{-\delta t} U_{\zeta}(x,t) \, dt, \label{E: linear rep} \\
		U_{\zeta}(x,t) &:= \int_{-\infty}^{\infty} Q^{\tilde{s}}_{\zeta}(x + (\tilde{s} - \langle x,e \rangle)e,t) \tilde{p}(\langle x,e \rangle, \tilde{s},t) d\tilde{s}, \nonumber \\
		Q^{\tilde{s}}_{\zeta}(x,t) &:= \int_{\langle e \rangle^{\perp}} \tilde{m}_{2}(\tilde{s},x + y + \zeta e) g(0,y,t) \, \mathcal{H}^{d-1}(dy). \nonumber
	\end{align}
(Above $\Delta_{\langle e \rangle}^{\perp}$ is the Laplacian in the $\langle e \rangle^{\perp}$ directions, that is, $\Delta_{\langle e \rangle^{\perp}} = \text{tr}( (\text{Id} - e \otimes e)D^{2})$.  Thus, $g(0,y,t) = (4 \pi t)^{-\frac{(d - 1)}{2}} \exp \left(- (4t)^{-1} \|y\|^{2}\right)$.)

We will prove that $\delta \tilde{v}^{\delta}_{\zeta} \to 0$ uniformly using the averaging induced by the diffusion in $\langle e \rangle^{\perp}$.  Toward that end, the following observation will play a decisive role.

		\begin{lemma} \label{L: averaging}  If $\mathcal{F}$ is a compact subset of $C(\mathbb{T}^{d})$ in the uniform norm topology, then there is a modulus $\eta : [0,\infty) \to [0,\infty)$ with $\lim_{\delta \to 0^{+}} \eta(\delta) = 0$ such that
			\begin{equation*}
				\sup \left\{ \left| \int_{\langle e \rangle^{\perp}} u(x + y) g(0,y,t) \, dt - \int_{\mathbb{T}^{d}} u(y) \, dy \right| \, \mid \, x \in \mathbb{T}^{d}, \, \, u \in \mathcal{F} \right\} \leq \eta (t^{-1}).
			\end{equation*}
		\end{lemma}  
		
			\begin{proof}  We only need to prove uniform convergence at a given $u \in C(\mathbb{T}^{d})$.  The uniformity in $\mathcal{F}$ then follows by the Arzel\`{a}-Ascoli Theorem.
			
			To start with, assume that $u \in C(\mathbb{T}^{d})$ is such that its Fourier series is summable, that is, $\sum_{k \in \mathbb{Z}^{d}} |\hat{u}(k)| < \infty$.  If we define $Q : \mathbb{T}^{d} \times (0,\infty) \to \mathbb{R}$ by 
				\begin{equation*}
					Q(x,t) = \int_{\langle e \rangle^{\perp}} u(x + y) g(0,y,t) \, \mathcal{H}^{d-1}(dy),
				\end{equation*}
			then an elementary computation shows that
				\begin{equation*}
					\hat{Q}(k,t) = \hat{u}(k) e^{- 4 \pi^{2} \|k - \langle k,e \rangle e\|^{2} t}.
				\end{equation*}
			Thus, the fact that $k \neq \langle k,e \rangle e$ for all $k \in \mathbb{Z}^{d}$ readily implies
				\begin{equation*}
					\lim_{t \to \infty} \sum_{k \in \mathbb{Z}^{d} \setminus \{0\}} |\hat{Q}(k,t)| = 0.
				\end{equation*}
			In particular, 
				\begin{equation*}
					\lim_{t \to \infty} \sup \left\{ \left|\int_{\langle e \rangle^{\perp}} u(x + y) g(0,y,t) \, \mathcal{H}^{d-1}(dy) - \int_{\mathbb{T}^{d}} u(y) \, dy \right| \, \mid \, x \in \mathbb{T}^{d} \right\} = 0.
				\end{equation*}  
				
			Finally, the general case follows by approximation.
			\end{proof}  
		
With Lemma \ref{L: averaging} in hand, Proposition \ref{P: convergence of fluctuations} follows readily:
	
	\begin{proof}[Proof of Proposition \ref{P: convergence of fluctuations}]  Recall that $V^{\delta}_{2} = \dot{q}^{-1} P^{\delta}_{2}$.  Hence our previous computations yield
		\begin{align*}
			\| \dot{q}^{-1}(\delta P^{\delta}_{2})\|_{L^{\infty}(\mathbb{R} \times \mathbb{T}^{d})} &= \sup \left\{ |\delta \tilde{v}_{\zeta}^{\delta}(x)| \, \mid \, x \in \mathbb{R}^{d}, \, \, \zeta \in \mathbb{R} \right\} \\
				&\leq \int_{0}^{\infty} \delta e^{-\delta t} \|Q^{\tilde{s}}_{\zeta}(\cdot,t)\|_{L^{\infty}(\mathbb{T}^{d})} \, dt.
		\end{align*}
	
	Observe that $\{\tilde{m}_{2}(\tilde{s},\cdot + \zeta e) \, \mid \, (\tilde{s},\zeta) \in \mathbb{R}^{2}\}$ is relatively compact in $C(\mathbb{T}^{d})$ by the choice of $\tilde{m}_{2}$.	Thus, by Lemma \ref{L: averaging}, there is a modulus $\eta : [0,\infty) \to [0,\infty)$ with $\lim_{\delta \to 0^{+}} \eta(\delta) = 0$ such that
		\begin{equation*}
			\sup \left\{ \|Q_{\zeta}^{\tilde{s}}(\cdot,t)\|_{L^{\infty}(\mathbb{T}^{d})} \, \mid \, (\tilde{s}, \zeta) \in \mathbb{R}^{2} \right\} \leq \eta(t^{-1}).
		\end{equation*}
	Putting it all together, we conclude by observing that
		\begin{equation*}
			\limsup_{\delta \to 0^{+}} \| \dot{q}^{-1}(\delta P^{\delta}_{2})\|_{L^{\infty}(\mathbb{R} \times \mathbb{T}^{d})} \leq \lim_{\delta \to 0^{+}} \int_{0}^{\infty} \delta e^{-\delta t} \eta(t^{-1}) \, dt = 0.
		\end{equation*}	 \end{proof}  
		
Now that Theorem \ref{T: approximate correctors} is proved, a few remarks are in order:

\begin{remark}  The diffusion in $\langle e \rangle^{\perp}$ directions is needed in Proposition \ref{P: convergence of fluctuations}.  More precisely, the same is no longer true if the forcing only depended on $s$.    

Here is an example.  Set $W(u) = (1 - u^{2})^{2}$ and consider the penalized correctors solving the following:
	\begin{equation*}
		q(s) \dot{q}(s) +\delta P^{\delta} - \ddot{P}^{\delta} + W''(q(s)) P^{\delta} = 0 \quad \text{in} \, \, \mathbb{R}.
	\end{equation*}
In this case, $q$ is an even function so $q \dot{q}$ is orthogonal to $\dot{q}$.  Thus, one can show that $\delta P^{\delta} \to 0$ uniformly as $\delta \to 0^{+}$.  However, this is no longer true when we renormalize by $\dot{q}$.

In this case, $v^{\delta} = \dot{q}^{-1} P^{\delta}$ solves the PDE:
	\begin{equation*}
		q(s) + \delta v^{\delta} - \ddot{v}^{\delta} + 2 \tanh(s) \dot{v}^{\delta} = 0 \quad \text{in} \, \, \mathbb{R}.
	\end{equation*}
For a given $\delta > 0$, as $\bar{s} \to \infty$, we find that $v^{\delta}(\cdot + \bar{s})$ converges to the bounded solution $v_{+}^{\delta}$ of
	\begin{equation*}
		1 + \delta v^{\delta}_{+} - \ddot{v}^{\delta}_{+} + 2 \dot{v}^{\delta}_{+} = 0 \quad \text{in} \, \, \mathbb{R}.
	\end{equation*} 
Since the coefficients are constant, this gives $-\delta v^{\delta}_{+} \equiv 1$.  In particular,
	\begin{equation*} 
		\left\| \dot{q}^{-1} (\delta P^{\delta}) \right\|_{L^{\infty}(\mathbb{R})} = \sup \left\{ \delta |v^{\delta}(s)| \, \mid \, s \in \mathbb{R} \right\} \geq 1 \quad \text{for all} \, \, \delta > 0.
	\end{equation*}

It is worth noting that, in the previous example, any bounded solution $P$ of $q(s) \dot{q}(s) - \ddot{P} + W''(q(s)) P = 0$ in $\mathbb{R}$ necessarily grows much faster than $\dot{q}$ as $s \to \pm \infty$.  That is, in this case, the function $\dot{q}^{-1} P$ is an unbounded solution of the associated PDE.
\end{remark}

Not only was the diffusion in orthogonal directions necessary in the proof of Theorem \ref{T: approximate correctors}, but irrationality was also:

\begin{remark} \label{R: degenerate directions}  If $e \in \mathbb{R} \mathbb{Z}^{d}$, then $(\dot{q}^{-1} \delta P^{\delta}_{2})_{\delta > 0}$ converges uniformly to a non-constant function in general.  To see this, notice that, if $e \in \mathbb{R} \mathbb{Z}^{d}$, then, in Lemma \ref{L: averaging}, the conclusion changes to the following one:
	\begin{equation*}
		\lim_{t \to \infty} \int_{\langle e \rangle^{\perp}} u(x + y,t) g(0,y,t) \, \mathcal{H}^{d-1}(dy) = \fint_{\mathbb{T}^{d-1}_{e}(\langle x,e \rangle)} u(\xi) \, \mathcal{H}^{d-1}(d \xi).
	\end{equation*}
Here the sub-tori $(\mathbb{T}^{d-1}_{e}(r))_{r \in [0,r_{e})}$ are defined by
	\begin{align*}
		\mathbb{T}^{d-1}_{e}(r) &= \left\{y \in \mathbb{T}^{d} \, \mid \, \langle y,e \rangle = r + \langle k,e\rangle \, \, \text{for some} \, \, k \in \mathbb{Z}^{d} \right\}, \\
		r_{e} &= \min \left\{ \langle k,e \rangle \, \mid \, k \in \mathbb{Z}^{d} \right\} \cap (0,\infty).
	\end{align*}
Geometrically, this becomes transparent, for example, when $e$ is a coordinate vector.

As a consequence of the previous observation, we see that if $(P^{\delta})_{\delta > 0}$ are the bounded solutions of $\tilde{m}(s,y) + \delta P^{\delta} + \mathcal{D}_{e}^{*} \mathcal{D}_{e} P^{\delta} + W''(q(s)) P^{\delta} = 0$ in $\mathbb{R} \times \mathbb{T}^{d}$ and $e \in \mathbb{R} \mathbb{Z}^{d}$, then 
	\begin{equation*}
		\lim_{\delta \to 0^{+}} \dot{q}(s)^{-1} (\delta P^{\delta}(s,y)) = \underline{\tilde{m}}_{e}(\langle y,e \rangle - s) \quad \text{uniformly in} \, \, \mathbb{R} \times \mathbb{T}^{d},
	\end{equation*}
where $\underline{\tilde{m}}_{e} : [0,r_{e}) \to \mathbb{R}$ is given by
	\begin{equation*}
		\underline{\tilde{m}}_{e}(\zeta) = c_{W}^{-1} \int_{-\infty}^{\infty} \fint_{\mathbb{T}^{d-1}_{e}(\zeta)} \tilde{m}(s,\xi) \dot{q}(s)^{2} \, \mathcal{H}^{d-1}(d \xi) \, ds.
	\end{equation*} 
While $\underline{\tilde{m}}_{e}$ certainly extends to a periodic function in $\mathbb{R}$, it need not be constant.

Similarly, \eqref{E: cell problem} cannot have a (e.g.\ weak) solution unless $\underline{\tilde{m}}_{e}$ is constant, and one can show that \eqref{E: approximate correctors} cannot hold unless $2\nu$ is larger than the oscillation of $\underline{\tilde{m}}_{e}$.
    \end{remark}  
    
Finally, we describe a situation where \eqref{E: cell problem} does have solutions in certain directions:

\begin{remark} \label{R: cell problem existence}  If we impose enough regularity assumptions on $m$ and arithmetic conditions on $e$, and if $\{Q^{\tilde{s}}_{\zeta}\}$ are defined as above, then it is possible to show the following estimate
	\begin{equation*}
		\sup \left\{ \int_{0}^{\infty} \|Q^{\tilde{s}}_{\zeta}(\cdot,t)\|_{L^{\infty}(\mathbb{T}^{d})} \, dt \, \mid \, (\tilde{s}, \zeta) \in \mathbb{R} \right\} < \infty.
	\end{equation*}
This can be made precise by following \cite[Proposition 22]{level set PDE}.  With this estimate, we use \eqref{E: linear rep} to see that $\|v_{\zeta}^{\delta}\|_{L^{\infty}(\mathbb{R}^{d})}$ is bounded independently of $(\delta,\zeta)$.  Therefore, we can send $\delta \to 0^{+}$ to obtain a solution of \eqref{E: cell problem}.  \end{remark}

\section{Irrational Contact Points}  \label{S: irrational directions}

The remainder of the paper is devoted to the proof of Proposition \ref{P: main proposition}.  This section establishes that the phase indicator function $\chi_{*}$ satisfies condition (a) in the definition of a supersolution in irrational directions (see Definition \ref{D: irrational directions}).   
 
Put another way, the goal of this section is to prove the following:

	\begin{prop} \label{P: graph localization} If $\varphi$ is a smooth function in $\mathbb{R}^{d} \times (0,\infty)$; $(x_{0},t_{0}) \in \mathbb{R}^{d} \times (0,\infty)$ is a point where $\chi_{*} - \varphi$ has a strict local minimum; and $D\varphi(x_{0},t_{0}) \in \mathbb{R}^{d} \setminus \mathbb{R} \mathbb{Z}^{d}$, then
		\begin{equation} \label{E: supersolution inequality}
			\overline{m}(\widehat{D\varphi}(x_{0},t_{0})) \varphi_{t}(x_{0},t_{0}) - \text{tr} \left( \left( \text{Id} - \widehat{D\varphi}(x_{0},t_{0}) \otimes \widehat{D\varphi}(x_{0},t_{0}) \right) D^{2} \varphi(x_{0},t_{0}) \right) \geq 0.
		\end{equation}
	\end{prop}  
	
The proof of Proposition \ref{P: graph localization} proceeds by contradiction and is divided into three steps.  The first step involves the construction of a suitable local subsolution of \eqref{E: effective equation}.  The second step, the so-called initialization step, shows that the solutions $(u^{\epsilon})_{\epsilon >0}$ develop a relatively sharp interface around $\{\varphi \approx 0\}$ after a short macroscopic time.  As in \cite{barles souganidis}, this initial step allows us to convert the macroscopic subsolution of the first step into a subsolution of \eqref{E: AC mobility}.  This conversion is precisely the third step.  If $\varphi$ does not satisfy \eqref{E: supersolution inequality}, these subsolutions slip underneath the solutions $(u^{\epsilon})_{\epsilon > 0}$ and force $(x_{0},t_{0})$ to be an interior point of the evolution $t \mapsto \Omega_{t}^{(1)}$, a contradiction.

\subsection{Macroscopic Subsolution}  \label{S: macroscopic subsolution} Here we recall some useful observations that follow from the assumption that $D\varphi(x_{0},t_{0}) \neq 0$.  It will be useful to introduce some notation.  

	Throughout the section, we let $e = \widehat{D\varphi}(x_{0},t_{0})$.  Let $\{e_{1},\dots,e_{d-1}\}$ be an orthonormal basis for $\mathbb{R}^{d-1}$ and $\mathcal{O}_{e} : \mathbb{R}^{d-1} \to \mathbb{R}^{d}$ be a linear isometry with $\mathcal{O}_{e}(\mathbb{R}^{d-1}) = \langle e \rangle^{\perp}$.  Given $R > 0$, we define the open cube $Q(0,R) \subseteq \mathbb{R}^{d-1}$ by 
		\begin{equation*}
			Q(0,R) = \{x' \in \mathbb{R}^{d-1} \, \mid \, \max\{ |\langle x', e_{1} \rangle|, \dots, |\langle x', e_{d-1} \rangle| \} < R/2 \}
		\end{equation*}
	For an $\tilde{x}' \in \mathbb{R}^{d-1}$, we set $Q(\tilde{x}',R) = \tilde{x}' + Q(0,R)$.  
	
	Using the coordinates determined by $\mathcal{O}_{e}$, we define open rectangular prisms $Q^{e}(0,R,\rho)$ of base length $R > 0$ and height ratio $\rho$ by 
		\begin{equation*}
			Q^{e}(0,R,\rho) = \left\{ \mathcal{O}_{e}(x') + s e \, \mid \, x' \in Q(0,R), \, \, s \in (-\rho R/2,\rho R/2) \right\}
		\end{equation*}
	Given $x \in \mathbb{R}^{d}$, we define $Q^{e}(x,R,\rho) = x + Q^{e}(0,R,\rho)$.
	
	Finally, an arbitrary point $x \in \mathbb{R}^{d}$ will frequently be written as $x = (x_{e},x')$ with the understanding that $x_{e} \in \mathbb{R}$ and $x' \in \mathbb{R}^{d-1}$ are such that $x = \mathcal{O}_{e}(x') + x_{e} e$.  In particular, we will write $x_{0} = (x_{0,e},x_{0}')$.
	
	We will use the fact that $\{\varphi = 0\}$ is locally a one-parameter family of graphs near $(x_{0},t_{0})$.  Specifically, we have

	\begin{prop} \label{P: graph thing}  There are constants $\rho, \nu, S, V> 0$ and a smooth function $g : Q(x_{0}',S) \times (t_{0} - \nu, t_{0} + \nu) \to \mathbb{R}$ such that
		\begin{itemize}
			\item[(i)] $\varphi(x,t) > 0$ (resp.\ $\varphi(x,t) \leq 0$) for some $(x,t) \in Q^{e}(x_{0},S,\rho) \times (t_{0} - \nu,t_{0} + \nu)$ if and only if $x_{e} > g(x',t)$ (resp.\ $x_{e} \leq g(x',t)$).
			\item[(ii)] $|g(x_{1}',t) - g(x_{2}',t)| \leq \frac{1}{2} \rho (d - 1)^{-\frac{1}{2}} \|x_{1}' - x_{2}'\|$ no matter the choice of $x_{1}',x_{2}' \in Q(x_{0}',S)$ or $t \in (t_{0} - \nu, t_{0} + \nu)$.
			\item[(iii)] $|g(x',t) - g(x',s)| \leq V |t - s|$ for all $x' \in Q(x_{0}',S)$ and $t, s \in (t_{0} - \nu, t_{0} + \nu)$.
		\end{itemize}
	Further, we can assume that $\rho < 1$ and $(x_{0},t_{0})$ is a strict local minimum of $\chi_{*} - \varphi$ in $Q^{e}(x_{0},S,\rho) \times (t_{0} - \nu,t_{0} + \nu)$.  
	\end{prop}
	
		\begin{proof}  The construction of $g$ is a classical application of the implicit function theorem.  The fact that $D\varphi(x_{0},t_{0}) = \|D\varphi(x_{0},t_{0})\| e$ implies that $Dg(x_{0},t_{0}) = 0$, and hence the existence of $\rho < 1$.  \end{proof}  
		
	By assumption, there is an $\tilde{\alpha} > 0$ such that
		\begin{equation*}
			\overline{m}(\widehat{D\varphi}(x_{0},t_{0})) \varphi_{t}(x_{0},t_{0}) - \text{tr} \left( \left( \text{Id} - \widehat{D\varphi}(x_{0},t_{0}) \otimes \widehat{D\varphi}(x_{0},t_{0}) \right) D^{2} \varphi(x_{0},t_{0}) \right) \leq -10 \tilde{\alpha} \|D\varphi(x_{0},t_{0})\|.
		\end{equation*}
	In particular, since $g$ is smooth, this implies there is an $0 < S_{1} < S$ and a $0 < \nu_{1} < \nu$ such that
		\begin{equation} \label{E: definition of S1}
			\overline{m}\left( \frac{e - Dg}{\sqrt{1 + \|Dg\|^{2}}} \right) g_{t} - \text{tr} \left( \left( \text{Id} - \frac{Dg \otimes Dg}{1 + \|Dg\|^{2}} \right) D^{2} g \right) \geq 9 \tilde{\alpha} \quad \text{in} \, \, Q(x_{0}',S_{1}) \times (t_{0} - \nu_{1},t_{0} + \nu_{1}).
		\end{equation}
		
	Next, given a free variable $c > 0$ to be determined, define $\tilde{g} : Q(x_{0}',S) \times (t_{0} - \nu, t_{0} + \nu) \to \mathbb{R}$ by 
		\begin{equation*}
			\tilde{g}(x',t) = g(x',t) + \frac{c \|x' - x_{0}'\|^{2}}{2}.
		\end{equation*}
	Notice that there is a $c_{1} > 0$ such that if $c \in (0,c_{1})$, then
		\begin{equation} \label{E: definition of c1}
			\overline{m}\left( \frac{e - D\tilde{g}}{\sqrt{1 + \|D\tilde{g}\|^{2}}} \right) \tilde{g}_{t} - \text{tr} \left( \left( \text{Id} - \frac{D\tilde{g} \otimes D\tilde{g}}{1 + \|D\tilde{g}\|^{2}} \right) D^{2} \tilde{g} \right) \geq 8 \tilde{\alpha} \quad \text{in} \, \, Q(x_{0}',S_{1}) \times (t_{0} - \nu_{1}, t_{0} + \nu_{1}).
		\end{equation}
		
	Finally, let $d : Q^{e}(x_{0},S_{1},\rho) \times (t_{0} - \nu_{1}, t_{0} + \nu_{1}) \to \mathbb{R}$ and $d_{c} : Q^{e}(x_{0},S_{1},\rho) \times (t_{0} - \nu_{1}, t_{0} + \nu_{1}) \to \mathbb{R}$ be the signed distance functions to $\{x_{e} = g(x',t)\}$ and $\{x_{e} = \tilde{g}(x',t)\}$, respectively.  Specifically, we define $d$ by
		\begin{equation*}
			d(x,t) = \left\{ \begin{array}{r l}
							\text{dist}(x, \{\tilde{x} \in Q^{e}(x_{0},S_{1},\rho) \, \mid \, \tilde{x}_{e} = g(\tilde{x}',t)\}), & \text{if} \, \, x_{e} > g(x',t), \\
							-\text{dist}(x, \{\tilde{x} \in Q^{e}(x_{0},S_{1},\rho) \, \mid \, \tilde{x}_{e} = g(\tilde{x}',t)\}), & \text{otherwise},
						\end{array} \right.
		\end{equation*}
	and we define $d_{c}$ similarly, but with $g$ replaced by $\tilde{g}$.
	
	Let $\alpha_{0} > 0$ be another free variable.  Arguing as in\ \cite[Appendix C]{pulsating einstein}, we deduce the following facts about $d_{c}$:
	
		\begin{lemma} \label{L: definition of S2}  Making $\nu_{1}$ smaller if necessary, there is a constant $\gamma > 0$ depending on $c_{1}$, $\varphi$, and $S_{1}$ but not on $c$ such that 
			\begin{align*}
				M &:= \sup \left\{ \|\partial_{t}^{j} D^{i}d_{c}(x,t)\| + \|\partial_{t}^{j} D^{i}d(x,t)\| \, \mid \, x \in Q^{e}(x_{0},S_{1}/2,\rho), \right. \\
					&\qquad \qquad \left. t \in (t_{0} - \nu_{1}, t_{0} + \nu_{1}), \, \, |d_{c}(x,t)| < \gamma, \, \, i + j \leq 4 \right\} < \infty.
			\end{align*}
		Furthermore, we can choose $0 < S_{1}' < S_{1}$ and $0 < \nu_{1}' < \nu_{1}$ in such a way that (making $\gamma$ smaller, if necessary)
			\begin{equation*}
				\overline{m}(Dd_{c}) d_{c,t} - \Delta d_{c} \leq -7 \tilde{\alpha}, \, \, |\overline{m}(Dd_{c}) - \overline{m}(e)| \leq \alpha_{0} \quad \text{in} \, \, Q^{e}(x_{0},S_{1}',\rho) \times (t_{0} - \nu_{1}',t_{0} - \nu_{1}') \cap \{|d| < \gamma\}.
			\end{equation*}
		\end{lemma}    

Finally, we let $\eta, \beta > 0$ be free variables to be determined and pick $0 < S_{2} < S_{1}'$ and $0 < \nu_{2} < \nu_{1}'$  and define $\tilde{d}_{c} : Q^{e}(x_{0},S_{2},\rho) \times (t_{0} - \nu_{2}, t_{0} + \nu_{2}) \to \mathbb{R}$ by 
	\begin{equation*}
		\tilde{d}_{c}(x,t) = d_{c}(x + \eta(t - (t_{0} - \nu_{2})) e, t)
	\end{equation*}
Notice that $\tilde{d}_{c}(\cdot,t)$ is the signed distance to the surface $\{x_{e} = \tilde{g}(x',t) - \eta(t - (t_{0} - \nu_{2}))\}$, and it is well-defined by making $S_{2}$ and $\nu_{2}$ smaller if necessary.

$\tilde{d}_{c}$ is the subsolution of \eqref{E: effective mobility} that will be used in the sequel.    Its key properties are summarized next.  In the statement, we use the notation $\partial_{p}$ for the parabolic boundary.  Specifically, for a space-time domain $Q \times (a,b)$ that means
	\begin{equation*}
		\partial_{p} [Q \times (a,b)] = \partial Q \times (a,b] \cup \overline{Q} \times \{a\}.
	\end{equation*} 

	\begin{prop} \label{P: key geometric proposition}  If $\eta \in (0,\Theta^{-1} \tilde{\alpha})$, $\beta \in (0,\frac{1}{2}\eta \nu_{2} \wedge 1 \wedge \frac{1}{2}S_{2})$, $c \in (0,c_{1})$, $\nu_{2} \in (0,\nu_{1}')$, and $S_{2} \in (0,S_{1}'/2)$ satisfy the inequalities
		\begin{equation} \label{E: key inequalities for good geometric behavior}
				(3 \eta + 2V) \nu_{2} < \frac{\rho S_{2}}{4} + \frac{c S^{2}_{2}}{8}
		\end{equation}
	then 
		\begin{align*}
			\overline{m}(D\tilde{d}_{c}) \tilde{d}_{c,t} - \Delta \tilde{d}_{c} \leq - 6 \tilde{\alpha}, \, \, |\overline{m}(D\tilde{d}_{c}) - \overline{m}(e)| \leq \alpha_{0} \quad \text{in} \, \, \{|\tilde{d}_{c}| < \gamma\} \\
			\chi_{\{\tilde{d}_{c} \geq \beta\}} \leq \chi_{\{d \geq \beta\}} \quad \text{on} \, \, \partial_{p} [Q^{e}(x_{0},S_{2},\rho) \times (t_{0} - \nu_{2},t_{0} + \nu_{2})]
		\end{align*}
	\end{prop}  
	
		\begin{proof}  The inequality $2 \eta \nu_{2} < \frac{\rho S_{2}}{2} <  \frac{\rho (S_{1}' - S_{2})}{2}$ implies that $x + \eta (t - (t_{0} - \nu_{2}))e \in Q^{e}(x_{0},S_{1}',\rho)$ whenever $x \in Q^{e}(x_{0},S_{2},\rho)$.  Thus, the first statement follows from Lemma \ref{L: definition of S2} and the inequality $\overline{m} \leq \Theta$.  
		
		Concerning the second statement, we know that $\tilde{d}_{c}(\cdot,t_{0} - \nu_{2}) = d_{c}(\cdot,t_{0} - \nu_{2})$.  Therefore, the ordering between $g$ and $\tilde{g}$ implies $\tilde{d}_{c} \leq d$ on the surface $\{t = t_{0} - \nu_{2}\}$.
		
		Next, we check the remaining inequalities, namely,
			\begin{equation*}
				\chi_{\{\tilde{d}_{c} \geq 2 \beta\}} \leq \chi_{\{d \geq 2 \beta\}} \quad \text{on} \, \, \partial Q^{e}(x_{0},S_{2},\rho) \times [t_{0} - \nu_{2},t_{0} + \nu_{2}]
			\end{equation*}
		We will start by examining points $x = (x_{e},x')$ with $x' \in \partial Q(x_{0},S_{2})$.
		
		Assume that $\tilde{d}_{c}(x,t) \geq \beta$ and $x' \in \partial Q(x_{0},S_{2})$.  To show that $d(x,t) \geq \beta$, we proceed point by point.  We claim that if $d(y,t) = 0$, then $\|y - x\| \geq \beta$.  If $\|y' - x'\| \geq \beta$, we are done.  Therefore, assume that $\|y' - x'\| \leq \beta$.  
		
		By the definition of $\tilde{d}_{c}$ and $d_{c}$, we have
			\begin{equation} \label{E: useful for tricky part}
				g(x',t) + \frac{c S_{2}^{2}}{8} - 2 \eta \nu_{2} \leq g(x',t) + \frac{c \|x' - x_{0}'\|^{2}}{2} - \eta(t - (t_{0} - \nu_{2})) < x_{e}.
			\end{equation}
		At the same time, $|g(y',t) - g(x',t)| \leq \frac{1}{2} \rho \beta$.  Therefore, since $y_{e} = g(y',t)$, we find
			\begin{equation*}
				y_{e} \leq g(x',t) + \frac{\rho \beta}{2} < x_{e} - \frac{c S_{2}^{2}}{8} + 2 \eta \nu_{2} + \frac{\beta}{2}.
			\end{equation*}
		Recalling that $2\beta < \eta \nu_{2}$ and appealing to \eqref{E: key inequalities for good geometric behavior}, we conclude
			\begin{equation*}
				\|y - x\| \geq |y_{e} - x_{e}| = x_{e} - y_{e} \geq \frac{c S_{2}^{2}}{8} - \left (2 \eta \nu_{2} + \frac{\beta}{2} \right) \geq \beta.
			\end{equation*}
		Hence $|d(x,t)| \geq \beta$.  Similarly, \eqref{E: useful for tricky part} shows that $g(x',t) < x_{e}$ so $d(x,t) > 0$.  Thus, $d(x,t) \geq \beta$ as claimed. 
		
		Finally, we consider points $(x,t)$ with $x$ on the top or bottom of the box $Q^{e}(x_{0},S_{2},\rho)$, that is, points for which $|x_{e} - x_{0,e}| = \rho S_{2}/2$.  To start with, observe that if $x_{e} = x_{0,e} - \rho S_{2}/2$, then $\tilde{d}_{c}(x,t) < 0 < \beta$.  Indeed, it suffices to show that $x_{e} + \eta (t - (t_{0} - \nu_{2})) < \tilde{g}(x',t)$, which is true since, by \eqref{E: key inequalities for good geometric behavior},
			\begin{align*}
				x_{e} + \eta(t - (t_{0} - \nu_{2})) &\leq g(x_{0}',t_{0}) - \frac{\rho S_{2}}{2} + 2 \eta \nu_{2} \\
					&\leq \tilde{g}(x',t) + 2 (V + \eta) \nu_{2} - \frac{\rho S_{2}}{4} < \tilde{g}(x',t).
			\end{align*}
			
	It remains to consider the case when $x_{e} = x_{0,e} + \rho S_{2}/2$.  We claim that $d(x,t) \geq \beta$ in this case.  To see this, we first show that if $d(y,t) = 0$ and $y \in Q^{e}(x_{0},S_{1},\rho)$, then $\|y - x\| \geq \beta$.  Indeed, in case $y \in Q^{e}(x_{0},S_{2} + 2\beta, \rho)$, we apply \eqref{E: key inequalities for good geometric behavior} to find
			\begin{align*}
				\|y - x\| &\geq |y_{e} - x_{e}| \\
					&= \left|(x_{e} - x_{0,e}) + (g(x_{0},t_{0}) - g(y',t)) \right| \\
					&\geq \frac{\rho S_{2}}{2} - |g(x_{0},t_{0}) - g(y',t)| \\
					&\geq \frac{\rho S_{2}}{2} - \frac{\rho (S_{2} + 2\beta)}{4} - 2 V \nu_{2} \geq \beta
			\end{align*}
		
	On the other hand, if $y \in Q^{e}(x_{0},S_{1},\rho) \setminus Q^{e}(x_{0},S_{2} + 2\beta,\rho)$, then we can consider two cases: (i) $y' \in Q(x_{0}',S_{1}) \setminus Q(x_{0}',S_{2} + 2 \beta)$ or (ii) $\rho (S_{2} + 2 \beta)/2 \leq |y_{e}| < \rho S_{1}/2$.  In case (i), $\|x - y\| \geq \|x' - y'\| \geq \beta$ follows immediately.  On the other hand, in case (ii), we can assume that both $|y_{e} - x_{0,e}| \geq \rho(S_{2} + 2 \beta)/2$ and $y' \in Q(x_{0}',S_{2} + 2 \beta)$, but then this contradicts Proposition \ref{P: graph thing} and $2V \nu_{2} < \rho S_{2}/2$.  So only case (i) is possible and then $|d(x,t)| \geq \beta$ follows.  
	
	Similarly, we find that $d(x,t) > 0$ so $d(x,t) \geq \beta$.\end{proof}  
		
In the remainder of this section, we will adjust the constants if necessary so that the hypotheses of Proposition \ref{P: key geometric proposition} hold.  In addition, we impose the following constraint on $\eta$:
	\begin{equation*}
		\eta \nu_{2} < \gamma.
	\end{equation*}
The justification for this restriction comes in the remark that follows.  Henceforth, $\eta$, $c$, $\nu_{2}$, and $S_{2}$ are fixed.  We reserve the right to make $\beta > 0$ smaller later.  Also note that $\alpha_{0}$ remains undetermined at this stage and so far no restrictions have been imposed on it.

\begin{remark} \label{R: trivial observations}  Notice that the boundary inequality in Proposition \ref{P: key geometric proposition} has the following (trivial) consequence: for each $\epsilon > 0$,
	\begin{equation*}
		(1 - \beta \epsilon) \chi_{\{\tilde{d}_{c} \geq \beta\}} - \chi_{\{\tilde{d}_{c} < \beta\}} \leq (1 - \beta \epsilon) \chi_{\{d \geq \beta\}} - \chi_{\{d < \beta\}} \quad \text{on} \, \, \partial_{p} [Q(x_{0},S_{2},\rho) \times (t_{0} - \nu_{2}, t_{0} + \nu_{2})].
	\end{equation*}
	
Further, notice that since $Dd_{c}(x_{0},t_{0}) = e$ and $\eta \nu_{2} < \gamma$, it follows that 
	\begin{equation*}
		\tilde{d}_{c}(x_{0},t_{0}) = d_{c}(x_{0} + \eta \nu_{2} e, t_{0}) = \eta \nu_{2} > 2\beta.
	\end{equation*}  
In particular, $\{\tilde{d}_{c} > 2\beta\}$ contains a neighborhood of $(x_{0},t_{0})$.
\end{remark}

\subsection{Initialization}  In this section, we prove an initialization result that shows that the solutions $(u^{\epsilon})_{\epsilon > 0}$ develop a sharp transition along the interface $\{\varphi = 0\}$ when $\epsilon > 0$ is sufficiently small.  Here we follow \cite{barles souganidis} using the result of Appendix \ref{A: initialization}.

	\begin{prop} \label{P: initialization}  Given $\delta \in (0,\frac{1}{2})$, there is a $\tau = \tau(\delta,\beta,\varphi) > 0$ and an $\epsilon_{0} = \epsilon_{0}(\beta,\varphi,u_{0}) > 0$ such that if $\epsilon \in (0,\epsilon_{0})$ and $t \in [t_{0} - \nu_{2},t_{0} + \nu_{2}]$, then
		\begin{equation*}
			u^{\epsilon}(\cdot,t + \tau \epsilon^{2} \log(\epsilon^{-1})) \geq (1 - \beta \epsilon) \chi_{\{d(\cdot,t) \geq \beta\}} - \chi_{\{d(\cdot,t) \leq \beta\}} \quad \text{in} \, \, \overline{Q^{e}(x_{0},S_{2},\rho)}.
		\end{equation*}
	\end{prop} 

		\begin{proof}  First of all, since $\{d \geq \beta\} \subseteq \{\chi_{*} = 1\}$, there is an $\epsilon_{0} > 0$ such that
			\begin{equation} \label{E: using containments}
				\{d \geq \beta\} \cap (Q^{e}(x_{0},S,\rho) \times [t_{0} - \nu,t_{0} - \nu]) \subseteq \{u^{\epsilon} > 1 - \delta\} \quad \text{if} \, \, \epsilon \in (0,\epsilon_{0}).
			\end{equation}
		For the rest of the proof, fix $t \in [t_{0} - \nu_{2},t_{0} + \nu_{2}]$.  
		
		Let $\underline{\psi} : \mathbb{R}^{d} \to \mathbb{R}$ be the function given by 
			\begin{equation*}
				\underline{\psi}(x) = \left\{ \begin{array}{r l} 
										1 - \delta, & x \in Q^{e}(x_{0},(S_{1} + S_{2})/2,\rho) \cap \{d(\cdot,t) \geq \beta/2\} \\
										-1, & \text{otherwise}
							\end{array} \right.
			\end{equation*}
		
		Fix a smooth, symmetric non-negative function $\rho$ such that $\rho = 0$ in $\mathbb{R}^{d} \setminus B(0,1)$ and $\int_{\mathbb{R}^{d}} \rho(x) \, dx = 1$.  Let $(\rho_{\nu})_{\nu > 0}$ be the mollifying family in $\mathbb{R}^{d}$ given by $\rho_{\nu}(x) = \nu^{-d} \rho(\nu^{-1} x)$.  In view of Lemma \ref{L: definition of S2}, there is a $\overline{\nu} > 0$ small and independent of $t$ with the following property:
			\begin{align*}
				x \in \{d \geq \beta\} \cap Q^{e}(x_{0},S_{2},\rho) \, \, \implies \, \, B(x,\overline{\nu}) \subseteq \{d(\cdot,t) \geq \beta/2\} \cap Q^{e}(x_{0},(S_{1} + S_{2})/2,\rho), \\
				x \in \{d \leq 0\} \cap Q^{e}(x_{0},S_{2},\rho) \, \, \implies \, \,  B(x,\overline{\nu}) \subseteq \{d(\cdot,t) \leq \beta/2\} \cap Q^{e}(x_{0},(S_{1} + S_{2})/2,\rho).
			\end{align*}
		Define $\psi = \rho_{\overline{\nu}} * \underline{\psi}$.  Recall that 	
			\begin{align} 
				\|D\psi\|_{L^{\infty}(\mathbb{R}^{d})} &\leq C_{\rho} \overline{\nu}^{-1}, \label{E: trivial gradient bound} \\
				\|D^{2}\psi\|_{L^{\infty}(\mathbb{R}^{d})} &\leq C_{\rho} \overline{\nu}^{-2}. \label{E: trivial hessian bound}
			\end{align}
		
		Notice that, by the choice of $\overline{\nu}$, if $x \in Q^{e}(x_{0},S_{2},\rho)$ and $d(x,t) \geq \beta$, then
			\begin{equation*}
				\psi(x) = \int_{B(x,\overline{\nu})} \psi(y) \rho_{\overline{\nu}}(x-y) \, dy = 1 - \delta.
			\end{equation*}  
		Similarly, if $x \in Q^{e}(x_{0},S_{2},\rho)$ and $d(x,t) \leq 0$, then $\psi(x) = -1$.  In summary,
			\begin{equation} \label{E: bump construction}
				\{d \geq \beta\} \cap Q^{e}(x_{0},S_{2},\rho) \subseteq \{\psi = 1 - \delta\}, \quad
				\{d \leq 0\} \cap Q^{e}(x_{0},S_{2},\rho) \subseteq \{\psi = -1\}
			\end{equation} 
			
		Let $\tilde{\chi}^{\epsilon}$ be the function from Appendix \ref{A: initialization}, Lemma \ref{L: universal ODE}; let $K > 0$ be a free variable; and define a family $(\underline{u}^{\epsilon})_{\epsilon > 0}$ in $\mathbb{R}^{d} \times [t,\infty)$ by 
			\begin{equation*}
				\underline{u}^{\epsilon}(x,t') = \tilde{\chi}^{\epsilon}( \psi(x) - \epsilon^{-1} K(t' - t), \epsilon^{-1} (t' - t)).
			\end{equation*}
		The construction of Appendix \ref{A: initialization} gives that $\tilde{\chi}^{\epsilon}_{s} \leq - \bar{f}(\chi^{\epsilon})$, where $\bar{f}$ is defined in \eqref{E: forcing}.  Thus, by arguing as in \cite[Lemma 4.1]{barles souganidis} and invoking \eqref{A: m positive} and \eqref{A: m bounded}, we find that $\underline{u}^{\epsilon}$ is a subsolution of \eqref{E: AC mobility} if $K$ is large enough.  (Notice that $K$ depends only on $\overline{\nu}$ through \eqref{E: trivial gradient bound} and \eqref{E: trivial hessian bound} and, thus, is independent of $t$).    Further, $\underline{u}^{\epsilon}(\cdot,0) = \psi \leq u^{\epsilon}(\cdot,t)$ by \eqref{E: using containments}.  Therefore, 
			\begin{equation*}
				\underline{u}^{\epsilon}(x,t') \leq u^{\epsilon}(x,t' + t) \quad \text{if} \, \, (x,t') \in \mathbb{R}^{d} \times [0,\infty).
			\end{equation*}
		Now the conclusion follows from the properties of $\chi$ arguing exactly as in \cite{barles souganidis}.  Note that $\tau$ depends only on $K$ and so is independent of $t$. 
		\end{proof}  
		
\subsection{Mesoscopic Subsolutions}  Finally, we use the macroscopic subsolution of Section \ref{S: macroscopic subsolution}, namely $\tilde{d}_{c}$, to build mesoscopic subsolutions of \eqref{E: AC mobility} that converge to $1$ in the sets $\{\tilde{d}_{c} > 2\beta\}$.  Appealing to Remark \ref{R: trivial observations}, we will then conclude that $u^{\epsilon} \to 1$ uniformly in a neighborhood of $(x_{0},t_{0})$, a patent contradiction.  

	Recall that $e = \widehat{D\varphi}(x_{0},t_{0})$.  Let $\alpha_{1}$ be one last free variable, for convenience.  Invoking Theorem \ref{T: approximate correctors}, we fix an approximate corrector $P_{e} = \overline{P}_{e} + P^{\delta}_{2} \in C^{2,\mu}(\mathbb{R} \times \mathbb{T}^{d})$ such that \eqref{E: approximate correctors} holds with $\nu = \alpha_{1}$.  

	Define a family $(v^{\epsilon})_{\epsilon > 0}$ in $\{(x,t) \in Q^{e}(x_{0},S_{2},\rho) \times [t_{0} - \nu_{2},t_{0} + \nu_{2}] \, \mid \, |\tilde{d}_{c}(x,t)| < \gamma\}$ by 
		\begin{equation*}
			v^{\epsilon}(x,t) = q \left( \frac{\tilde{d}_{c}(x,t) - 2 \beta}{\epsilon} \right) + \epsilon \left( \tilde{d}_{c,t}(x,t) P_{e} \left( \frac{\tilde{d}_{c}(x,t) - 2 \beta}{\epsilon}, \frac{x}{\epsilon} \right) - 2 \beta \right).
		\end{equation*}
	We show below that, provided $\alpha_{0}$, $\alpha_{1}$, and $\beta$ are chosen appropriately, $v^{\epsilon}$ is a subsolution of \eqref{E: AC mobility} as soon as $\epsilon > 0$ is small enough.  
	
	In order to invoke the comparison principle, we extend $(v^{\epsilon})_{\epsilon > 0}$ to $Q^{e}(x_{0},S_{2},\rho) \times (t_{0} - \nu_{2},t_{0} + \nu_{2})$.  The construction again follows \cite{barles souganidis}.  First, we define $(\overline{v}^{\epsilon})_{\epsilon > 0}$ by
		\begin{equation*}
			\overline{v}^{\epsilon}(x,t) = \left\{ \begin{array}{r l}
										\max \left\{v^{\epsilon}(x,t), -1\right\}, & \tilde{d}_{c}(x,t) \geq - (\gamma + 2\beta)/2 \\
										-1, & \tilde{d}_{c}(x,t) < - (\gamma + 2\beta)/2
									\end{array} \right.
		\end{equation*}
	Finally, we fix a smooth, non-decreasing function $f : \mathbb{R} \to [0,1]$ such that
		\begin{equation*}
			f(\xi) = 1 \quad \text{if} \quad \xi \geq \frac{7 \gamma}{8} + \frac{\beta}{4}, \quad f(\xi) = 0 \quad \text{if} \, \, \quad \xi \leq \frac{3 \gamma}{4} + \frac{\beta}{2}
		\end{equation*} 
	and define $(w^{\epsilon})_{\epsilon > 0}$ by
		\begin{equation*}
			w^{\epsilon}(x,t) = (1 - f(\tilde{d}_{c}(x,t))) \overline{v}^{\epsilon}(x,t) + f(\tilde{d}_{c}(x,t)) (1 - \beta \epsilon)
		\end{equation*}
		
	Here is the main result we will need to proceed:
	
		\begin{prop} \label{P: mesoscopic subsolution}  There is an $\epsilon_{1} > 0$ and a choice of the parameters $\alpha_{0}$, $\alpha_{1}$, and $\beta$ such that $w^{\epsilon}$ satisfies
			\begin{equation*}
				\left\{ \begin{array}{r l}
						m(\epsilon^{-1} x, \widehat{Dw^{\epsilon}}) w^{\epsilon}_{t} - \Delta w^{\epsilon} + \epsilon^{-2} W'(w^{\epsilon}) \leq 0 & \text{in} \, \, Q^{e}(x_{0},S_{2},\rho) \times (t_{0} - \nu_{2},t_{0}] \\
						w^{\epsilon} \leq (1 - \beta \epsilon) \chi_{\{\tilde{d}_{c} \geq \beta\}} - \chi_{\{\tilde{d}_{c} < \beta\}} & \text{on} \, \, \partial_{p} [Q^{e}(x_{0},S_{2},\rho) \times (t_{0} - \nu_{2},t_{0}]]
				\end{array} \right.
			\end{equation*}
		Furthermore, if $(x,t) \in Q^{e}(x_{0},S_{2},\rho) \times (t_{0} - \nu_{2},t_{0}]$ satisfies $\tilde{d}_{c}(x,t) > 2 \beta$, then 
			\begin{equation*}
				\liminf \nolimits_{*} w^{\epsilon}(x,t) = 1.
			\end{equation*}
		\end{prop}  
		
	Before proving Proposition \ref{P: mesoscopic subsolution}, let us use it to prove Proposition \ref{P: graph localization}.  
	
	\begin{proof}[Proof of Proposition \ref{P: graph localization}] By Propositions \ref{P: initialization} and \ref{P: mesoscopic subsolution} and Remark \ref{R: trivial observations}, for each $\epsilon \in (0,\epsilon_{0} \wedge \epsilon_{1})$, we have
		\begin{equation*}
			w^{\epsilon} \leq u^{\epsilon}(\cdot, \cdot + \tau \epsilon^{2} \log(\epsilon^{-1})) \quad \text{on} \, \, \partial_{p} [Q^{e}(x_{0},S_{2},\rho) \times (t_{0} - \nu_{2},t_{0})].
		\end{equation*}
	According to Remark \ref{R: trivial observations} and Proposition \ref{P: mesoscopic subsolution}, there is a small $r > 0$ such that
		\begin{equation*}
			1 = \liminf \nolimits_{*} w^{\epsilon}(x,t) \quad \text{if} \, \, \|x - x_{0}\| + |t - t_{0}| < r.
		\end{equation*}
	Therefore, for such points $(x,t)$,
		\begin{equation*}
			\liminf \nolimits_{*} u^{\epsilon}(x,t) \geq \liminf \nolimits_{*} w^{\epsilon}(x,t) = 1.
		\end{equation*}
	Hence $\chi_{*} = 1$ in a neighborhood of $(x_{0},t_{0})$.  This contradicts the assumption that $D\varphi(x_{0},t_{0}) \neq 0$.
	 \end{proof}
	 
	 Now we proceed with the proof of Proposition \ref{P: mesoscopic subsolution}.  The proof will be presented through a series of lemmas.  The first deals with $v^{\epsilon}$ near the interface.
	 
	 \begin{lemma} \label{L: delicate subsolution property}  There is a choice of $\beta$, $\alpha_{1}$, and $\alpha_{0}$ such that: (i) $\beta$ is small enough to satisfy the constraints of the previous section and (ii) there is a constant $\nu(\beta,\tilde{\alpha})$ such that, for all $\epsilon > 0$ small enough, we have
	 	\begin{equation*}
			m(\epsilon^{-1} x, \widehat{Dv^{\epsilon}}) v^{\epsilon}_{t} - \Delta v^{\epsilon} + \epsilon^{-2} W'(v^{\epsilon}) \leq - \frac{\nu(\beta,\tilde{\alpha})}{3\epsilon} \quad \text{in} \, \, \{|d| < \gamma\}.
		\end{equation*}
	\end{lemma}  
	
The selection of $\beta$, $\alpha_{1}$, and $\alpha_{0}$ below is a little delicate.  The reason is, at some stage, the fact that $D\tilde{d}_{c}$ is not constant introduces errors.  Let $\omega_{e}$ be a modulus of continuity for $m$ and $\overline{m}$ at $e$, that is,
	\begin{equation*}
		\omega_{e}(\chi) = \sup \left\{ |m(y,v) - m(y,e)| + |\overline{m}(v) - \overline{m}(e)| \, \mid \, \|v - e\| \leq \chi, \, \, y \in \mathbb{T}^{d} \right\}.
	\end{equation*}
The errors $D\tilde{d}_{c}$ are proportional to $\omega_{e}(\alpha_{0})$ with proportionality constants depending on the choice of $\alpha_{1}$ through the magnitudes of the derivatives of $P_{e}$.  Therefore, to control these errors, we need to choose $\alpha_{1}$ before $\alpha_{0}$.  Given that $\beta$ depends on $\alpha_{0}$ through Proposition \ref{P: key geometric proposition}, it has to be chosen last.  

		\begin{proof}  In what follows, to declutter the notation, it will be convenient to define $s = s(x,t) = \tilde{d}_{c}(x,t) - 2 \beta$ and $p_{\epsilon}$ by 
			\begin{equation*}
				p_{\epsilon}(x,t) = \epsilon \mathcal{D}_{Dd(x,t)} P_{e}(\epsilon^{-1}s(x,t),\epsilon^{-1}x) + \epsilon^{2} D\tilde{d}_{c,t}(x,t) P_{e}(\epsilon^{-1}s(x,t),\epsilon^{-1}x).
			\end{equation*}
		By definition of $v^{\epsilon}$, the regularity of $P_{e}^{\delta}$, and the definition of $M$ in Lemma \ref{L: definition of S2}, we have
			\begin{align*}
				m(\epsilon^{-1} x, \epsilon Dv^{\epsilon}) v^{\epsilon}_{t} &= \epsilon^{-1} \dot{q}(\epsilon^{-1} s) m(\epsilon^{-1} x, \dot{q}(\epsilon^{-1} s)D\tilde{d}_{c}(x,t) + p_{\epsilon}) \tilde{d}_{c,t}(x,t) + O(1) \\
					&\leq \epsilon^{-1} \dot{q}(\epsilon^{-1} s) \left( m(\epsilon^{-1} x, \dot{q}(\epsilon^{-1} s)e) \tilde{d}_{c,t}(x,t) + M\omega_{e}(\alpha_{0} + O(\epsilon))\right) + O(1).
			\end{align*}
		Note, in addition, that, no matter the choice of $e' \in S^{d-1}$, the function $P^{\delta}_{2} = V^{\delta}_{2} \dot{q}$ defined in \eqref{E: fluctuation cell problem} satisfies
			\begin{equation*}
				\mathcal{D}^{*}_{e'} \mathcal{D}_{e'} P^{\delta}_{2} + W''(q(s)) P^{\delta}_{2} = \dot{q}(s)  \left( \mathcal{D}^{*}_{e'} \mathcal{D}_{e'} V^{\delta}_{2} - \frac{2 \ddot{q}(s)}{\dot{q}(s)} \langle e', \mathcal{D}_{e'} V_{2}^{\delta} \rangle \right) \quad \text{in} \, \, \mathbb{R} \times \mathbb{T}^{d}.
			\end{equation*}
		Combining this with the estimate $\|D\tilde{d}_{c} - e\| \leq \alpha_{0}$ from Proposition \ref{P: key geometric proposition}, we find
			\begin{align*}
				- \Delta v^{\epsilon} &= - \epsilon^{-2} \ddot{q}(\epsilon^{-1}s) - \epsilon^{-1} \dot{q}(\epsilon^{-1}s) \Delta \tilde{d}_{c}(x,t) + \epsilon^{-1} \mathcal{D}^{*}_{D\tilde{d}_{c}(x,t)} \mathcal{D}_{D\tilde{d}_{c}(x,t)} P_{e}(\epsilon^{-1}s,\epsilon^{-1}x) \tilde{d}_{c,t}(x,t) + O(1) \\
					&\leq - \epsilon^{-2} \ddot{q}(\epsilon^{-1}s) + \epsilon^{-1} \dot{q}(\epsilon^{-1}s) \left( [\overline{m}(e) - m(\epsilon^{-1}x,\dot{q}(\epsilon^{-1}s)e)] \tilde{d}_{c,t}(x,t) - \Delta \tilde{d}_{c}(x,t) + M\alpha_{1} + O(\alpha_{0})\right)  \\
					&\qquad -\epsilon^{-1} W''(q(\epsilon^{-1}s)) P_{e}+ O(1).
			\end{align*}
		Therefore, writing $\epsilon^{-2} W'(v^{\epsilon}) = \epsilon^{-2} W'(q) + \epsilon^{-1} W''(q) P^{\delta} - 2 \beta \epsilon^{-1} W''(q) + O(1)$ yields
			\begin{align*}
				m(\epsilon^{-1} x, \epsilon Dv^{\epsilon}) v^{\epsilon}_{t} - \Delta v^{\epsilon} + \epsilon^{-2} W'(v^{\epsilon}) &= \epsilon^{-1} \Big \{ \dot{q}(\epsilon^{-1}s,\epsilon^{-1}x) \Big ( \overline{m}(D\tilde{d}_{c}(x,t)) \tilde{d}_{c,t}(x,t) - \Delta \tilde{d}_{c}(x,t) \\
				&\qquad + 2 M\omega_{e}(\alpha_{0} + O(\epsilon)) + O(\alpha_{0}) + M\alpha_{1} \Big ) \\
				&\qquad -2 \beta W''(q(\epsilon^{-1}s)) \Big \} + O(1).
			\end{align*}
		Thus, appealing once more to Proposition \ref{P: key geometric proposition}, 
			\begin{align*}
				&m(\epsilon^{-1} x, \epsilon Dv^{\epsilon}) v^{\epsilon}_{t} - \Delta v^{\epsilon} + \epsilon^{-2} W'(v^{\epsilon}) \\
				&\qquad \leq \epsilon^{-1} \left\{ - (6 \tilde{\alpha} + 2M\omega_{e}(\alpha_{0} + O(\epsilon)) + O(\alpha_{0}) + M \alpha_{1}) \dot{q}(\epsilon^{-1}s) - 2 \beta \epsilon W''(q(\epsilon^{-1}s)) \right\} + O(1).
			\end{align*}
		
		Finally, we choose $\alpha_{1}$, $\alpha_{0}$, and $\beta$, in that order.  To start with, choose $\alpha_{1}$ so that $M \alpha_{1} < \tilde{\alpha}$.  Next, choose $\alpha_{0}$ so that, in the expression above, as soon as $\epsilon$ is small enough, we have
			\begin{equation*}
				2 M\omega_{e}(\alpha_{0} + O(\epsilon)) + O(\alpha_{0}) \leq 2 M\omega_{e}(2\alpha_{0}) + O(\alpha_{0}) \leq \tilde{\alpha}.
			\end{equation*} 
		Note that this choice depends on $\alpha_{1}$ through the magnitude of the derivatives of $V^{\delta}_{2}$, which contribute to the $O(\alpha_{0})$ term.  However, it does not depend on any of the parameters introduced in Section \ref{S: macroscopic subsolution} (and, in particular, introduces no new restrictions on $\beta$) so there is no risk of circular reasoning.  
		
		By \cite[Lemma 4.3]{barles souganidis}, there is a $\overline{\beta}(\tilde{\alpha}) > 0$ such that if $\beta \in (0,\overline{\beta}(\tilde{\alpha}))$, then
			\begin{equation} \label{E: BS trick}
				\nu(\beta,\tilde{\alpha}) := \sup \left\{ 3 \tilde{\alpha} \dot{q}(s) + 2 \beta W''(q(s)) \, \mid \, s \in \mathbb{R} \right\} > 0.
			\end{equation}
		At last, fix such a $\beta$ consistently with the restrictions of Section \ref{S: macroscopic subsolution}.  Note that, with this choice of $(\alpha_{0},\alpha_{1},\beta)$, for small enough $\epsilon > 0$, we find
			\begin{equation*}
				m(\epsilon^{-1} x, \epsilon Dv^{\epsilon}) v^{\epsilon}_{t} - \Delta v^{\epsilon} + \epsilon^{-2} W'(v^{\epsilon}) \leq - \frac{\nu(\beta,\tilde{\alpha})}{2 \epsilon} + O(1) \quad \text{in} \, \, \{|\tilde{d}_{c}| < \gamma\}.
			\end{equation*}
	\end{proof}  
	
	Henceforth, we assume that $\beta$, $\alpha_{0}$, and $\alpha_{1}$ have been chosen so that Lemma \ref{L: delicate subsolution property} holds.  These three parameters will remain fixed throughout the rest of this section.
		
	Next, we show that the functions $(\overline{v}^{\epsilon})_{\epsilon > 0}$ are subsolutions away from $\{\tilde{d}_{c} > 0\}$.  
	
	\begin{lemma} \label{L: first stage subsolution proof} If $\epsilon > 0$ is small enough, then $\overline{v}^{\epsilon}$ is a subsolution of \eqref{E: AC mobility} in $\{\tilde{d}_{c} < \gamma\}$.  \end{lemma}  
	
		\begin{proof}  It is clear that $\overline{v}^{\epsilon}$ is a subsolution in $\{- \frac{(\gamma + 2 \beta)}{2} < \tilde{d}_{c} < \gamma \}$, being the maximum of two subsolutions there.  At the same time, we claim that $\overline{v}^{\epsilon}$ is a subsolution in $\{\tilde{d}_{c} < -(\frac{3 \beta}{2} + \frac{\gamma}{4})\}$ as soon as $\epsilon > 0$ is small enough.  Indeed, if $\tilde{d}_{c}(x,t) < -(\frac{3 \beta}{2} + \frac{\gamma}{4})$, then the exponential estimates of Propositions \ref{P: standing wave} and \ref{P: 1D corrector} and Corollary \ref{C: cell problem solution} imply that
			\begin{equation*}
				v^{\epsilon}(x,t) \leq q \left( - \frac{\gamma}{4 \epsilon} \right) + \epsilon \left( M \left| P_{e} \left( \frac{\tilde{d}_{c}(x,t) - 2 \beta}{\epsilon} \right)\right| - 2 \beta \right) \leq -1 + C \exp \left(-\frac{\gamma}{4 C \epsilon}\right) - 2 \beta \epsilon.
			\end{equation*}
		Hence $\overline{v}^{\epsilon} = -1$ in $\{\tilde{d}_{c} < -(\frac{3 \beta}{2} + \frac{\gamma}{4})\}$ as soon as $\epsilon$ is sufficiently small.  In particular, that makes $\overline{v}^{\epsilon}$ is a subsolution in $\{\tilde{d}_{c} < \gamma\}$.    \end{proof}  
		
	Finally, we verify that $(w^{\epsilon})_{\epsilon > 0}$ remains a subsolution inside $\{\tilde{d}_{c} > 0\}$ and has the right boundary behavior.
		
	\begin{lemma} \label{L: first initialization}  If $\epsilon > 0$ is small enough, then $w^{\epsilon}$ is a subsolution of \eqref{E: AC mobility} in $Q^{e}(x_{0},S_{2},\rho) \times (t_{0} - \nu_{2},t_{0} + \nu_{2})$ and 
		\begin{equation*}
			w^{\epsilon} \leq (1 - \beta \epsilon) \chi_{\{\tilde{d}_{c} \geq \beta\}} - \chi_{\{\tilde{d}_{c} < \beta\}} \quad \text{on} \, \, \partial_{p} [Q^{e}(x_{0},S_{2},\rho) \times (t_{0} - \nu_{2},t_{0}]].
		\end{equation*}
	\end{lemma}  
	
		\begin{proof}  Arguing as in Lemma \ref{L: first stage subsolution proof}, we see that $\overline{v}^{\epsilon} = v^{\epsilon}$ in $\{\tilde{d}_{c} > \frac{\gamma + 2 \beta}{2}\}$ as soon as $\epsilon > 0$ is sufficiently small.  In fact, we can assume that $1 - v^{\epsilon} \leq 2 \beta \epsilon$ in $\{\tilde{d}_{c} > \frac{\gamma + 2 \beta}{2}\}$.  
		
		Plugging $w^{\epsilon}$ into the equation in $\{ \frac{\gamma + 2 \beta}{2} < \tilde{d}_{c} < \gamma\}$, we find
	\begin{align*}
		m \left(\epsilon^{-1}x, \widehat{Dw^{\epsilon}} \right) w^{\epsilon}_{t} - \Delta w^{\epsilon} + \epsilon^{-2} W'(w^{\epsilon}) &= m \left(\epsilon^{-1} x, \widehat{Dw^{\epsilon}} \right) w_{t}^{\epsilon} - (1 - f(\tilde{d}(x,t))) \Delta v^{\epsilon} \\
		&\quad + 2 f'(\tilde{d}(x,t)) \langle D\tilde{d}, Dv^{\epsilon} \rangle + \epsilon^{-2} W'(w^{\epsilon})  \\
		&\quad + \left( f'(\tilde{d}(x,t)) \Delta \tilde{d} + 2f''(\tilde{d}(x,t)) \right) \\
		&\quad \times (v^{\epsilon}(x,t) - (1 - \beta \epsilon)) \\
		&= (I) + (II) + (III) + (IV) + (V) + (VI)
	\end{align*}
where, in view of the exponential estimates in Propositions \ref{P: standing wave} and \ref{P: 1D corrector} and Corollary \ref{C: cell problem solution}, the error terms can be estimated as follows:
	\begin{align*}
		(I) &= (1 - f(\tilde{d}(x,t))) \{m(\epsilon^{-1} x, \widehat{Dv^{\epsilon}}) v^{\epsilon}_{t} - \Delta v^{\epsilon} + \epsilon^{-2} W'(v^{\epsilon})\} \leq - \frac{1}{3}(1 - f(\tilde{d}(x,t))) \nu(\beta,\tilde{\alpha}) \epsilon^{-1}, \\
		(II) &= f(\tilde{d}(x,t)) \epsilon^{-2} W'(1 - \beta \epsilon) \leq - f(\tilde{d}(x,t)) \beta W''(1) \epsilon^{-1} + O(1) \\
		(III) &= \epsilon^{-2} W'((1 - f) v^{\epsilon} + f (1 - \beta \epsilon)) - (1 - f) \epsilon^{-2} W'(v^{\epsilon}) - f \epsilon^{-2} W'(1 - \beta \epsilon) \\
			&= \epsilon^{-2} O(|v^{\epsilon} - 1|^{2} + \epsilon^{2}) = O(1) \\
		(IV) &= (1 - f(\tilde{d}(x,t))) (m(\epsilon^{-1} x, \widehat{Dw^{\epsilon}}) - m(\epsilon^{-1} x, \widehat{Dv^{\epsilon}})) v^{\epsilon}_{t} \\
			&\leq C \left[ \epsilon^{-1} \dot{q} \left(\frac{\tilde{d}(x,t) - 2 \beta}{\epsilon}\right) + \left| \partial_{s} P_{e} \left(\frac{\tilde{d}(x,t) - 2 \beta}{\epsilon} \right) \right| + \epsilon \right] \leq C \exp \left(- (C \epsilon)^{-1}\left(\frac{\gamma - 2 \beta}{2}\right) \right) \\
		(V) &= (1 - f(\tilde{d}(x,t))) m(\epsilon^{-1} x, \widehat{Dw^{\epsilon}}) (w^{\epsilon}_{t} - v^{\epsilon}_{t}) \\
			&\leq C \exp \left(- (C \epsilon)^{-1}\left(\frac{\gamma - 2 \beta}{2}\right) \right) + C f'(\tilde{d}(x,t)) \tilde{d}_{t} |v^{\epsilon}(x,t) - (1 - \beta \epsilon)| \\
			&\leq C \left[ \exp \left(- (C \epsilon)^{-1}\left(\frac{\gamma - 2 \beta}{2} \right) \right) + \epsilon \right] \\
		(VI) &= 2 f'(\tilde{d}(x,t)) \langle D\tilde{d}, Dv^{\epsilon} \rangle + \left( f'(\tilde{d}(x,t)) \Delta \tilde{d} + 2f''(\tilde{d}(x,t)) \right) (v^{\epsilon}(x,t) - (1 - \beta \epsilon)) \\
			&\leq C \left( \left(\epsilon^{-1} + 1\right) \exp \left( - (C\epsilon)^{-1} \left(\frac{\gamma - 2 \beta}{2} \right) \right) \right) + C\epsilon
	\end{align*}
In particular, we find, in the limit $\epsilon \to 0^{+}$,
	\begin{equation*}
		m \left(\epsilon^{-1}x, \widehat{Dw^{\epsilon}} \right) w^{\epsilon}_{t} - \Delta w^{\epsilon} + \epsilon^{-2} W'(w^{\epsilon}) \leq - \min\left\{\frac{1}{3}\nu(\beta,\tilde{\alpha}), \beta W''(1)\right\} \epsilon^{-1} + O(1).
	\end{equation*}
Thus, $w^{\epsilon}$ is a subsolution in the domain $\{\frac{\gamma + 2 \beta}{2} < \tilde{d}_{c} < \gamma \}$ as soon as $\epsilon$ is small enough.  At the same time, $w^{\epsilon} = \overline{v}_{\epsilon}$ in $\{\tilde{d}_{c} < \frac{3\gamma}{4} + \frac{\beta}{2}\}$ and $w^{\epsilon} = 1 - \beta \epsilon$ in $\{\frac{7 \gamma}{8} + \frac{\beta}{4} < \tilde{d}_{c}\}$  so $w^{\epsilon}$ is actually a subsolution in $Q^{e}(x_{0},S_{2},\rho) \times (t_{0} - \nu_{2},t_{0} + \nu_{2}]$. 

Finally, we check the boundary condition.  We claim that, for all $\epsilon > 0$ small enough,
	\begin{equation} \label{E: crazy inequality}
		v^{\epsilon} \leq (1 - \beta \epsilon)\chi_{\{\tilde{d}_{c} \geq \beta\}} - \chi_{\{\tilde{d}_{c} < \beta\}} \quad \text{in} \, \, Q^{e}(x_{0},S_{2},\rho) \times (t_{0} - \nu_{2},t_{0} + \nu_{2}).	
	\end{equation}
To see this, first, choose $\kappa > 0$ such that 
	\begin{equation*}
		\max \left\{ \dot{q}(s) \, \mid \, s \geq \kappa \right\} < \beta \Theta^{-1}.
	\end{equation*}
Now notice that if $\tilde{d}_{c}(x,t) \leq 2 \beta + \kappa \epsilon$, then	
	\begin{equation*}
		v^{\epsilon}(x,t) \leq q(\kappa) + \epsilon \Theta \|\dot{q}\|_{L^{\infty}(\mathbb{R})} - 2 \beta \epsilon
	\end{equation*}
while $\tilde{d}_{c}(x,t) > 2 \beta + \kappa \epsilon$ implies, by the choice of $\kappa$,
	\begin{equation*}
		v^{\epsilon}(x,t) \leq 1 - \beta \epsilon.
	\end{equation*}
Thus, there is an $\bar{\epsilon} > 0$ such that, for each $\epsilon \in (0,\bar{\epsilon})$,
	\begin{equation*}
		v^{\epsilon} \leq 1 - \beta \epsilon \quad \text{in} \, \, Q^{e}(x_{0},S_{2},\rho) \times (t_{0} - \nu_{2},t_{0} + \nu_{2}).
	\end{equation*}
	
Finally, if $\tilde{d}_{c}(x,t) < \beta$, then, making $\bar{\epsilon} > 0$ smaller if necessary, we find, for each $\epsilon \in (0,\bar{\epsilon})$,
	\begin{equation*}
		v^{\epsilon}(x,t) \leq -1 + C \exp \left( - \frac{\beta}{C \epsilon} \right) - 2 \beta \epsilon \leq -1.
	\end{equation*}
This completes the proof of \eqref{E: crazy inequality}.  Since $f(\xi) = 0$ if $\xi \leq 2\beta$, the claimed boundary behavior of $w^{\epsilon}$ follows.
\end{proof}  
%
%

\section{Rational Contact Points}  \label{S: rational directions}

In this section, we prove the analogue of Proposition \ref{P: graph localization} for rational directions assuming in addition that the level set of $\varphi$ is nearly flat at the contact point.  That is, we tackle condition (b) in Definition \ref{D: irrational directions}.  The main result is stated below:
	
	\begin{prop} \label{P: rational directions}  Fix $\delta \in (0,1)$.  If $\varphi$ is a smooth function in $\mathbb{R}^{d} \times (0,\infty)$; $(x_{0},t_{0}) \in \mathbb{R}^{d} \times (0,\infty)$ is a point where $\chi_{*} - \varphi$ has a strict local minimum; $D\varphi(x_{0},t_{0}) \in \mathbb{R} \mathbb{Z}^{d} \setminus \{0\}$; and the level set of $\varphi$ has is $\delta$-flat at $(x_{0},t_{0})$ in the following sense
		\begin{equation*}
			\left\| \left(\text{Id} - \widehat{D\varphi}(x_{0},t_{0}) \otimes \widehat{D\varphi}(x_{0},t_{0}) \right) D^{2} \varphi(x_{0},t_{0}) \right\| \leq \delta \|D\varphi(x_{0},t_{0})\|
		\end{equation*}
	then
		\begin{equation} \label{E: supersolution inequality rational case}
			\varphi_{t}(x_{0},t_{0}) \geq -10\theta^{-1} \delta \|D\varphi(x_{0},t_{0})\|
		\end{equation}
	\end{prop}  
	
The proof of Proposition \ref{P: rational directions} is a minor modification of the proof of Proposition \ref{P: graph localization}.  Let us summarize the details.

Again, proceed by contradiction.  If \eqref{E: supersolution inequality rational case} fails, then we can construct $\tilde{d}_{c}$ once more in such a way that
	\begin{equation*}
		\tilde{d}_{c,t} \leq - 9 \theta^{-1} \delta \quad \text{in} \, \, \{|\tilde{d}_{c}| < \gamma\}
	\end{equation*}
Further, by continuity, we can make $S_{2}$, $\nu_{2}$, $c$, and $\gamma$ so small that
	\begin{equation*}
		|\Delta \tilde{d}_{c}| \leq 2 \delta \quad \text{in} \, \, \{|\tilde{d}_{c}| < \gamma\}
	\end{equation*}
Hence, with these changes, the conclusions of Proposition \ref{P: key geometric proposition} still hold except $-6 \tilde{\alpha}$ should be replaced by $-6\delta$ and the mobility $\overline{m}$ by the constant $\theta$.

The construction of mesoscopic subsolutions proceeds as before, except this time $v^{\epsilon}$ is simply given by
	\begin{equation*}
		v^{\epsilon}(x,t) = q\left(\frac{\tilde{d}_{c}(x,t) - 2 \beta}{\epsilon} \right) - 2 \beta \epsilon.
	\end{equation*}
When it comes time to check that $v^{\epsilon}$ is a subsolution, we use
	\begin{equation*}
		m(\epsilon^{-1} x, D\tilde{d}_{c}) \tilde{d}_{c,t} - \Delta \tilde{d}_{c} \leq \theta \tilde{d}_{c,t} - \Delta \tilde{d}_{c} \leq - 6\delta \quad \text{in} \, \, \{|\tilde{d}_{c}| < \gamma\}.
	\end{equation*}
The remainder of the construction goes through exactly as before.
	
\section{Shrinking Subsolutions}  \label{S: zero normal}

In this section, we construct mesoscopic subsolutions of \eqref{E: AC mobility} that approximate characteristic functions of shrinking balls.  Using these, we prove that the limiting evolution satisfies the remaining differential inequality in Definition \ref{D: irrational directions}, namely condition (c).  Employing similar ideas, we also prove the necessary inclusions relating the macroscopic phases to the initial datum.

\subsection{Finite Speed of Shrinking}

As shown in \cite{barles souganidis}, to prove $\chi_{*}$ satisfies the right differential inequality when the gradient vanishes, it suffices to check that balls contained in $\{\chi_{*} = 1\}$ cannot shrink too fast.  Toward that end, we begin by proving the next result:

	\begin{prop} \label{P: differential_relation}  Fix $R > 0$ and $t_{0} \geq 0$ and assume that $B(x_{0},R) \subseteq \Omega_{t_{0}}^{(1)}$.  Given $\underline{\theta} \in (0, \theta)$, there is an $h > 0$ depending continuously on $R$ (and independent of $(x_{0},t_{0})$) such that 
		\begin{equation*}
			B\left(x_{0},\sqrt{R^{2} - 2 \underline{\theta}^{-1} (d-1) s}\right) \subseteq \Omega_{t_{0} + s}^{(1)} \quad \text{if} \, \, s \in [0,h)
		\end{equation*}  
	\end{prop} 
	
By invoking the proposition, we can prove

	\begin{theorem} \label{T: ball_inclusion}  Fix $R > 0$ and $t_{0} > 0$ and assume that $B(x_{0},R) \subseteq \Omega_{t}^{(1)}$.  If $\underline{\theta} < \theta$, then 
		\begin{equation*}
			B \left(x_{0},\sqrt{R^{2} - 2\underline{\theta}^{-1} (d - 1)s}\right) \subseteq \Omega_{t_{0} + s}^{(1)} \quad \text{for each} \, \, s \in \left[0,\frac{\underline{\theta}R^{2}}{2 (d -1)}\right].
		\end{equation*}    
	\end{theorem}  
	
		\begin{proof}  For each $s \in [0,\frac{\underline{\theta}R^{2}}{2 (d -1)}]$, define $R: (0,\infty) \to [0,\infty)$ by
			\begin{equation*}
				R(s) = \sup \left\{ r \geq 0 \, \mid \, B(x_{0},r) \subseteq \Omega_{t_{0} + s}^{(1)}\right\}.
			\end{equation*}
		Note that the definition of $R$ implies that $B(x_{0},R(s)) \subseteq \Omega^{(1)}_{t_{0} + s}$.  Moreover, by assumption, $R(0) \geq R$.  Let $T = \inf \left\{s > 0 \, \mid \, R(s) = 0\right\}$.
		
		We claim that $s \mapsto R(s)$ is a lower semi-continuous viscosity supersolution of the ODE	
			\begin{equation*}
				\underline{\theta} \dot{R} + \frac{(d -1)}{R} \geq 0 \quad \text{in} \, \, (0,T).
			\end{equation*}
		The lower semi-continuity follows from the fact that $\chi_{*}$ is lower semi-continuous.
		
		Notice that, given an $s \in (0,T)$, Proposition \ref{P: differential_relation} yields an $h > 0$ such that if $s' \in (s- h,s)$, then
			\begin{equation*}
				R(s) \geq \sqrt{R(s')^{2} - \frac{2 (d - 1) (s - s')}{\underline{\theta}}}.
			\end{equation*}
		From this, it follows easily that if $\varphi$ is a smooth function and $s' \mapsto R(s') - \varphi(s')$ has a local minimum at $s$, then 
			\begin{equation*}
				\underline{\theta} \dot{\varphi}(s) + \frac{(d - 1)}{R(s)} \geq 0.
			\end{equation*}
		
		By the comparison principle for viscosity solutions, we deduce that $s \mapsto R(s)$ is at least as large as the solution of the ODE with initial condition $R(0)$.  In particular,
			\begin{equation*}
				R(s) \geq \sqrt{R(0) - \frac{2(d - 1)s}{\underline{\theta}}} \geq \sqrt{R - \frac{2(d - 1)s}{\underline{\theta}}} \quad \text{in} \, \, (0,T).
			\end{equation*}
		Note, in addition, that this inequality yields $T \geq \frac{\underline{\theta}R^{2}}{2 (d -1)}$.   \end{proof}  
	
Now we prove Proposition \ref{P: differential_relation}.  First, observe that the function $d : \mathbb{R}^{d} \times [0,\frac{\underline{\theta}R^{2}}{2 (d -1)}] \to \mathbb{R}$ given by 
	\begin{equation*}
		d(x,t) = \sqrt{R^{2} - 2 \underline{\theta}^{-1} (d-1) t} - \|x\|
	\end{equation*} 
satisfies
	\begin{align*}
		\theta d_{t} - \text{tr} \left( \left(\text{Id} - \widehat{Dd} \otimes \widehat{Dd} \right) D^{2} d \right) &= -\frac{\theta (d - 1)}{\underline{\theta} \sqrt{R^{2} - 2 \underline{\theta}^{-1}(d - 1)(t - t_{0})}} + \frac{(d - 1)}{\|x\|}.
	\end{align*}
Note, in addition, that $d_{t} \leq 0$.  A direction computation yields the following lemma:
	
	\begin{lemma}  Fix $R > 0$ and $t_{0} > 0$.  For each $\rho \in (0,1)$ and $\nu \in (0,\frac{\theta}{\underline{\theta}} - 1)$, the function $d$ above satisfies 
		\begin{equation*}
			\theta d_{t} - \text{tr} \left( \left(\text{Id} - \widehat{Dd} \otimes \widehat{Dd} \right) D^{2} d \right) \leq - \frac{1}{R} \left[ \frac{\theta}{\underline{\theta}} - 1 - \nu \right] \quad \text{in} \, \, A_{\rho,\nu} \times \left(0,\frac{R^{2} \underline{\theta}}{2 (d - 1)}\right)
		\end{equation*}
	where $A_{\rho,\nu} = B(0,(1 - \rho)^{-1}R) \setminus \overline{B(0,(1 + \nu)^{-1} R)}$.  \end{lemma}  
	
		
We use $d$ to construct global mesoscopic subsolutions arguing as in Section \ref{S: irrational directions}.  To start with, define $v^{\epsilon} : A_{\rho,\nu} \times \left(0,\frac{R^{2} \underline{\theta}}{2 (d - 1)}\right) \to \mathbb{R}$ by 
	\begin{equation*}
		v^{\epsilon}(x,t) = q \left( \frac{d(x,t) - 2 \beta}{\epsilon}\right) - 2 \beta \epsilon
	\end{equation*}
Observe that, using the sign of $d_{t}$, we can compute
	\begin{align*}
		m(\epsilon^{-1}x, \widehat{Dv^{\epsilon}}) v^{\epsilon}_{t} - \Delta v^{\epsilon} + \epsilon^{-2} W'(v^{\epsilon}) &= \epsilon^{-1} m(\epsilon^{-1} x, Dd(x,t)) \dot{q} d_{t} - \epsilon^{-2} \ddot{q} - \epsilon^{-1} \dot{q} \Delta d  \\
					&\quad + \epsilon^{-2} W'(q) - 2 \beta \epsilon^{-1} W''(q) \\
					&\leq \epsilon^{-1} \dot{q} \left( \theta d_{t} - \Delta d \right) - 2 \beta \epsilon^{-1} W''(q) \\
					&\leq - \epsilon^{-1} \left( C_{R} \dot{q} + 2\beta W''(q) \right)
	\end{align*}
where $C_{R} = \frac{1}{R} \left[ \frac{\theta}{\theta'} - 1 - \nu \right] > 0$.  As in \cite{barles souganidis}, we can choose $\overline{\beta} = \overline{\beta}(\nu) > 0$ so that, for each $\beta \in (0,\overline{\beta})$,
	\begin{equation*}
		\mu_{\beta} := \min \left\{ C_{R} \dot{q}(s) + 2 \beta W''(q(s)) \, \mid \, s \in \mathbb{R} \right\} > 0
	\end{equation*}
This gives
	\begin{equation*}
		m(\epsilon^{-1}x, \widehat{Dv^{\epsilon}}) v^{\epsilon}_{t} - \Delta v^{\epsilon} + \epsilon^{-2} W'(v^{\epsilon}) \leq - \mu_{\beta} \epsilon^{-1} \quad \text{in} \, \, A_{\rho,\nu} \times \left(0,\frac{R^{2} \underline{\theta}}{2 (d - 1)}\right)
	\end{equation*}
	
We will not be able to proceed in the entire time interval $\left(0,\frac{R^{2} \underline{\theta}}{2 (d - 1)}\right)$ since the interface $\{d = 0\}$ does not remain in $A_{\rho,\nu}$.  Therefore, we restrict attention to $\mathbb{R}^{d} \times (0,T)$ for some $T > 0$ and choose $\gamma > 0$ so that
	\begin{equation*}
		\{(x,t) \in \mathbb{R}^{d} \times (0,T) \, \mid \, |d(x,t)| < \gamma\} \subseteq A_{\rho,\nu} \times [0,T].
	\end{equation*}
Clearly, it is possible to do this by continuity.  A concrete choice of $T$ and $\gamma$ is
	\begin{equation*}
		T = \frac{\underline{\theta}}{4(d - 1)} \cdot \frac{R^{2} \nu (\nu + 2)}{(\nu + 1)^{2}}, \quad \gamma = \left[\frac{R \nu}{2 (\nu + 1)}\right] \wedge \left[\frac{\rho R}{2(1 - \rho)}\right].
	\end{equation*} 
Notice that, for a fixed $(\underline{\theta},\rho,\nu)$, $T$ and $\gamma$ depend continuously on $R$.   
	
Next, we define $(\overline{v}^{\epsilon})_{\epsilon > 0}$ and $(w^{\epsilon})_{\epsilon > 0}$ in $\mathbb{R}^{d} \times [0,T]$ as before with the choice of $\gamma$ just selected.  To get things started, we need the following variant of Lemma \ref{L: first initialization}:

	\begin{lemma} \label{L: initialization_balls}  There is a $\tau > 0$ depending only on $\beta$ and an $\epsilon_{0} > 0$ such that, for each $\epsilon \in (0,\epsilon_{0})$,
		\begin{equation*}
			u^{\epsilon}(\cdot, t_{0} + \tau \epsilon^{2} \log(\epsilon^{-1})) \geq (1 - \beta \epsilon) \chi_{\{\|x\| \leq R - \beta\}} - \chi_{\{\|x\| > R - \beta\}} \quad \text{in} \, \, \mathbb{R}^{d}.
		\end{equation*}
	\end{lemma}  
	
The proof follows by arguing exactly in \cite[Lemma 4.1]{barles souganidis}, replacing the function $\chi$ used there by the same function $\tilde{\chi}^{\epsilon}$ used in Lemma \ref{L: first initialization}.  
	
Comparing $u^{\epsilon}$ and $w^{\epsilon}$ as in Section \ref{S: irrational directions}, we find
	\begin{equation*}
		w^{\epsilon}(x,t) \leq u^{\epsilon}(x,t + \tau \epsilon^{2} \log(\epsilon^{-1})) \quad \text{if} \, \, (x,t) \in \mathbb{R}^{d} \times [0,T].
	\end{equation*}
Combined with the fact that $\liminf_{*} w^{\epsilon}(x,t) = 1$ if $d(x,t) \geq 2 \beta$, this gives
	\begin{equation*}
		\{d(\cdot,t) \geq 2 \beta\} \subseteq \Omega_{t}^{(1)} \quad \text{if} \, \, t \in [0,T].
	\end{equation*}
Sending $\beta \to 0^{+}$, we obtain the conclusion of Proposition \ref{P: differential_relation} with $h = T$.

\subsection{Supersolution property at zero}  The finite shrinking speed of the previous section is intimately related to the final differential inequality in Definitions \ref{D: irrational directions} and \ref{D: barles georgelin}.  In fact, it implies it, as shown in the next result.

	\begin{prop} \label{P: zero normal}  If $\varphi : \mathbb{R}^{d} \times (0,\infty) \to \mathbb{R}$ is smooth, $\chi_{*} - \varphi$ has a strict local minimum at $(x_{0},t_{0}) \in \mathbb{R}^{d} \times (0,\infty)$, and $\|D\varphi(x_{0},t_{0})\| = \|D^{2}\varphi(x_{0},t_{0})\| = 0$, then
		\begin{equation*}
			\varphi_{t}(x_{0},t_{0}) \geq 0
		\end{equation*} 
	\end{prop} 
	
	The proof below is based on an insight from \cite{barles souganidis}.
	
		\begin{proof}  First, notice that if $\chi_{*}(x_{0},t_{0}) = 1$, then $\chi_{*} = 1$ in a neighborhood of $(x_{0},t_{0})$, and this implies $\varphi_{t}(x_{0},t_{0}) = 0$ directly.  
		
		Assume now that $\chi_{*}(x_{0},t_{0}) = -1$ and, without loss of generality, that $\varphi(x_{0},t_{0}) = 0$.  It follows that there is an open ball $B \subseteq \mathbb{R}^{d} \times (0,\infty)$ containing $(x_{0},t_{0})$ such that 
			\begin{equation*}
				\chi_{*}(x,t) - \varphi(x,t) \geq -1 \quad \text{if} \, \, (x,t) \in B
			\end{equation*}
		In particular, since $\varphi(x_{0},t_{0}) = 0$, this gives
			\begin{equation} \label{E: key_point}
				\chi_{*}(x,t) - \varphi_{t}(x_{0},t_{0})(t - t_{0}) + o(\|x - x_{0}\|^{2} + |t - t_{0}|) \geq -1 \quad \text{if} \, \, (x,t) \in B
			\end{equation}
			
		Let $C = \frac{4 (d - 1)}{\underline{\theta}}$.  We claim that there is a sequence $((x_{n},t_{n}))_{n \in \mathbb{N}}$ such that 
			\begin{align*}
				(x_{0},t_{0}) &= \lim_{n \to \infty} (x_{n},t_{n}), \quad t_{n} < t_{0}, \\
				\|x_{n} - x_{0}\|^{2} &\leq C |t_{n} - t_{0}|, \quad \chi_{*}(x_{n},t_{n}) = -1
			\end{align*}
		Assuming the claim is true, we set $(x,t) = (x_{n},t_{n})$ in \eqref{E: key_point} to find
			\begin{equation*}
				\varphi_{t}(x_{0},t_{0})(t_{0} - t_{n}) + o(|t_{n} - t_{0}|) \geq 0.
			\end{equation*}
		Dividing by $t_{0} - t_{n}$ and sending $n \to \infty$, this yields
			\begin{equation*}
				\varphi_{t}(x_{0},t_{0}) \geq 0
			\end{equation*}
			
		It remains to prove the claim.  We argue by contradiction, assuming that it is false.  We can then fix an $s \in(0,t_{0})$ such that $B(x_{0},\sqrt{C(t_{0} - s)}) \subseteq \Omega_{s}^{(1)}$.  Now Theorem \ref{T: ball_inclusion} implies that 
			\begin{equation*}
				B\left(x_{0},\sqrt{C(t_{0} - s) - \frac{2(d-1) (t - s)}{\underline{\theta}}}\right) \subseteq \Omega_{t}^{(1)} \quad \text{if} \, \, t \in \left[0, \frac{C \underline{\theta}(t_{0} - s)}{2(d - 1)} \right] + s
			\end{equation*}  
		At the same time, notice that, by the choice of $C$,
			\begin{equation*}
				s + \frac{C \underline{\theta}(t_{0} - s)}{2 (d -1)} = s + 2(t_{0} - s) > t_{0}
			\end{equation*}
		Thus, we deduce that
			\begin{equation*}
				x_{0} \in B\left(x_{0},\sqrt{C(t_{0} - s) - \frac{2(d-1) (t_{0} - s)}{\underline{\theta}}}\right) \subseteq \Omega_{t_{0}}^{(1)},
			\end{equation*}
		but this contradicts the assumption that $\chi_{*}(x_{0},t_{0}) = -1$.  \end{proof}   
		
\subsection{Initial datum}  The proof of Proposition \ref{P: differential_relation} can be modified slightly to prove that $\Omega_{0}^{(1)} \supseteq \{u_{0} > 0\}$ as claimed in Proposition \ref{P: main proposition}.  

	\begin{prop} \label{P: initial data}  $\Omega_{0}^{(1)} \supseteq \{u_{0} > 0\}$.   \end{prop}  
	
To prove this, we will use the following variant of Lemma \ref{L: initialization_balls}.

	\begin{lemma} Given $\beta, r \in (0,1)$ and $x_{0} \in \mathbb{R}^{d}$, if $B(x_{0},r) \subseteq \{u_{0} > 0\}$, then there is a $\tau > 0$ depending only on $\beta$ and $r$ and $\epsilon_{0} \in (0,1)$ such that, for each $\epsilon \in (0,\epsilon_{0})$,
		\begin{equation*}
			u^{\epsilon}(\cdot, \tau \epsilon^{2} \log(\epsilon^{-1})) \geq (1 - \beta \epsilon) \chi_{\{\|x - x_{0}\| \leq r - \beta\}} - \chi_{\{\|x - x_{0}\| > r - \beta\}}
		\end{equation*}
	\end{lemma}   
	
	Now we prove the proposition.
	
		\begin{proof}[Proof of Proposition \ref{P: initial data}]  Suppose $B(x_{0},r) \subseteq \{u_{0} > 0\}$.  Fix $\beta \in (0,r)$.  By the lemma,
			\begin{equation*}
				\liminf \nolimits_{*} u^{\epsilon}(x,0) = 1 \quad \text{if} \, \, x \in B(x_{0},r -\beta).
			\end{equation*}
		Sending $\beta \to 0^{+}$, we conclude that $B(x_{0},r) \subseteq \Omega_{0}^{(1)}$.  It follows that $\{u_{0} > 0\} \subseteq \Omega_{0}^{(1)}$.  
		
		Suppose $\{u_{0} = 0\}$ has empty interior.  Replacing $\chi^{*}$ by $\chi_{*}$ and subsolutions by supersolutions, we find that $\Omega_{0}^{(2)} \supseteq \{u_{0} < 0\}$.  Since $\Omega_{0}^{(1)} \cap \Omega_{0}^{(2)} = \phi$ and $\{u_{0} = 0\}$ has empty interior, it follows that $\Omega_{0}^{(1)} = \{u_{0} > 0\}$ and $\Omega_{0}^{(2)} = \{u_{0} < 0\}$.      \end{proof}  

\appendix

\section{Initialization} \label{A: initialization}

We are interested in proving an ``initialization" type result for the phase field equation
	\begin{equation*}
				m(x,\widehat{Du}) u_{t} - \Delta u + W'(u) = 0 \quad \text{in} \, \, \mathbb{R}^{d} \times (0,\infty).
	\end{equation*} 
Defining the modified forcing $\bar{f} : [-1,1] \to \mathbb{R}$ by 
	\begin{equation*}
		\bar{f}(u) = \left\{ \begin{array}{r l}
								\Theta^{-1} W'(u), & \text{if} \, \, u \in [0,1] \\
								\theta^{-1} W'(u), & \text{if} \, \, u \in [-1,0] 
						\end{array} \right.
	\end{equation*}
	we will construct a ``universal subsolution" $\tilde{\chi}^{\epsilon}$ of the ODE
	\begin{equation*}
		\tilde{\chi}^{\epsilon}_{t} \leq -\bar{f}(\tilde{\chi}^{\epsilon}) \quad \text{in} \, \, [-1,1]
	\end{equation*}
with the same properties as the specific subsolution used in \cite[Lemma 4.1]{barles souganidis}.  As shown in Proposition \ref{P: initialization} above, the choice of $\bar{f}$ enables us to build certain subsolutions of \eqref{E: AC mobility} that are used to analyze the development of sharp interfaces as $\epsilon \to 0^{+}$.

In this section, we build $\tilde{\chi}^{\epsilon}$ by proceeding by analogy with Chen's paper \cite[Section 3]{chen}.  The main result is Lemma \ref{L: universal ODE} below.

\subsection{Preliminaries}  The assumptions on $W$ imply that we can fix a $\mu \in (0,\frac{1}{8})$ such that
	\begin{align}
		\frac{W''(-1)}{2} \leq &W''(u) \leq \frac{3 W''(-1)}{2} \quad \text{if} \, \, u \in [-1 - 2 \mu, -1 + 2 \mu] \label{E: second_deriv_neg_side} \\
		\frac{W''(1)}{2} \leq &W''(u) \leq \frac{3 W''(1)}{2} \quad \text{if} \, \, u \in [1 - 2 \mu, 1] \nonumber \\
		\frac{3 W''(0)}{2} \leq &W''(u) \leq \frac{W''(0)}{2} \quad \text{if} \, \, u \in [-\mu, \mu] \nonumber
	\end{align}
From this, we deduce that
	\begin{align*}
		\theta^{-1}W'(u) \leq &W'(u) \leq \Theta^{-1} W'(u) \quad \text{if} \, \, u \in [0,1] \cup [-1 - 2 \mu,-1], \\
		\Theta^{-1} W'(u) \leq &W'(u) \leq \theta^{-1} W'(u) \quad \text{if} \, \, u \in [-1,0] \cup [1, 1 + 2\mu].
	\end{align*}
	
\subsection{Regularization}  Fix $\epsilon \in (0,1)$.  To start with, let $\bar{f} : [-3, 3] \to \mathbb{R}$ be the function defined by 
	\begin{equation} \label{E: forcing}
		\bar{f}(u) = \left\{ \begin{array}{r l}
								\Theta^{-1} W'(u), & \text{if} \, \, u \in [0,3] \cup [-3,-1] \\
								\theta^{-1} W'(u), & \text{if} \, \, u \in [-1,0] 
						\end{array} \right.
	\end{equation}
Notice that $\bar{f}$ is Lipschitz continuous.

Next, we modify $\bar{f}$.  To start with, let $\rho : [-1,1] \to [0,\infty)$ be a smooth function with $\rho(s) = \rho(-s)$, $\rho(1) = \rho(-1) = 0$, and $\int_{-1}^{1} \rho(s) \, ds = 1$, and, given $\epsilon \in (0,1)$, define $\rho^{\epsilon}(s) = \epsilon^{-1} \rho(\epsilon^{-1} s)$.  Define $\bar{f}_{\epsilon}$ by $\bar{f}_{\epsilon} = \rho^{\epsilon} * \bar{f} + 2 \text{Lip}(\bar{f}; [-3,3]) \epsilon$.  Notice that, by construction, for each $u \in [-2,2]$, we have
	\begin{align*}
		\bar{f}_{\epsilon}(u) - \bar{f}(u) &\geq \text{Lip}(\bar{f};[-3,3]) \epsilon
	\end{align*}
In particular, $\bar{f}_{\epsilon} \geq \bar{f}$ in $[-2,2]$.  

Finally, fix a cut-off function $\eta \in C^{\infty}(\mathbb{R}; [0,1])$ such that
	\begin{align*}
		\eta(u) = 1 \, \, \text{if} \, \, u \in \left[-\frac{1}{4},\infty\right),& \quad \eta(u) = 0 \, \, \text{if} \, \, u \in \left(-\infty,-\frac{3}{4}\right] \\
		|\eta'(u)| &\leq 4
	\end{align*}	
and define $f_{\epsilon} : \mathbb{R} \to \mathbb{R}$ by 
	\begin{equation*}
		f_{\epsilon}(u) = \eta(u) \bar{f}(u) + (1 - \eta(u)) \bar{f}_{\epsilon}(u)
	\end{equation*}
	
Some properties of $f_{\epsilon}$ are summarized next:

	\begin{lemma} \label{L: regularization}  There is a constant $M_{1} > 0$ and an $\epsilon_{0} > 0$ with $\epsilon_{0} < \frac{1}{2}$ such that if $\epsilon \in (0,\epsilon_{0})$, then
		\begin{itemize}
			\item[(i)] $\frac{W''(-1)}{2 \Theta} \leq f_{\epsilon}'(u) \leq \frac{3W''(-1)}{2 \theta}$ if $u \in [-1 - \mu, -1 + \mu]$
			\item[(ii)] There is a $z_{\epsilon} \in [-1 - \mu,-1]$ such that 
				\begin{equation} \label{E: zeros}
					\{u \in [-1 - \mu,1] \, \mid \, f_{\epsilon}(u) = 0\} = \{z_{\epsilon},0,1\}
				\end{equation}
			\item[(iii)] $\|f'_{\epsilon}\|_{L^{\infty}([-(1 + \mu),1])} \leq M_{1}$
		\end{itemize}
	\end{lemma}  
	
	The lemma follows directly from the properties of $\bar{f}$ and the definition of $\bar{f}_{\epsilon}$.  Therefore, the proof is omitted.

\subsection{Modification}  In what follows, let $M = \theta^{-1} \|W''\|_{L^{\infty}([-3,3])}$.  Following \cite[Section 3]{chen}, we now fix a family of cut-off functions $(\zeta_{\epsilon})_{\epsilon \in (0,1)} \subseteq C^{\infty}_{c}(\mathbb{R}; [0,1])$ such that, for each $\epsilon \in (0,1)$,
	\begin{itemize}
		\item[(a)] $\zeta_{\epsilon}(u) = 1$ if $u \in [0,2 \epsilon |\log(\epsilon)|]$,
		\item[(b)] $\zeta_{\epsilon}(u) = 0$ if $u \in (-\infty,-\frac{\epsilon}{M}] \cup [3 \epsilon |\log(\epsilon)|, \infty)$,
		\item[(c)] $\zeta_{\epsilon}'$ satisfies the bounds
			\begin{align*}
				0 \leq \zeta_{\epsilon}'(u) \leq \frac{2 M}{\epsilon} \qquad \text{if} \, \, u \in [-\frac{\epsilon}{M}, 0], \quad
				 -\frac{2}{\epsilon |\log(\epsilon)|} \leq \zeta_{\epsilon}'(s) \leq 0 \qquad \text{if} \, \, u \in [0,3\epsilon |\log(\epsilon)|].
			\end{align*}
	\end{itemize}
Now we define $\tilde{f}_{\epsilon} : \mathbb{R} \to \mathbb{R}$ by 
	\begin{equation*}
		\tilde{f}_{\epsilon}(u) = (1 - \zeta_{\epsilon}(u)) f_{\epsilon}(u) + \zeta_{\epsilon}(u) \left(\frac{\epsilon |\log(\epsilon)| - u}{|\log(\epsilon)|} \right)
	\end{equation*}
	
To start with, we record some properties of the family $(\tilde{f}_{\epsilon})_{\epsilon \in (0,1)}$: 

	\begin{lemma}  \label{L: modification}There are positive constants $c, M_{2}, \epsilon_{1} > 0$ such that if $\epsilon \in (0,\epsilon_{0} \wedge \epsilon_{1})$, then
		\begin{itemize}
			\item[(a)] $\tilde{f}_{\epsilon} \geq f_{\epsilon} \geq \bar{f}$ in $[-1 - \mu, 1]$.
			\item[(b)] The following inequalities hold away from $0$:
				\begin{align*}
					\tilde{f}_{\epsilon}(u) \leq - c \epsilon \quad \text{if} \quad u \in [2 \epsilon |\log(\epsilon)|, 3 \epsilon |\log(\epsilon)|], \quad \tilde{f}_{\epsilon}(u) &\geq c \epsilon \quad \text{if} \quad u \in \left[- \frac{\epsilon}{M},0\right].
				\end{align*}
			\item[(c)] $\|\tilde{f}_{\epsilon}'\|_{L^{\infty}([-1 - \mu,1])} \leq M_{2}$.
		\end{itemize}
	\end{lemma}  
		
			\begin{proof}  To see that (a) holds, observe that the identity $f_{\epsilon}(0) = \bar{f}(0) = 0$ implies we can write
				\begin{equation} \label{E: dreaded_formula}
					\tilde{f}_{\epsilon}(u) = f_{\epsilon}(u) + \zeta_{\epsilon}(u) \left(\epsilon - \left( \frac{1}{|\log(\epsilon)|} + \frac{f_{\epsilon}(u) - f_{\epsilon}(0)}{u} \right) u \right)
				\end{equation}
			Recall that the $\zeta_{\epsilon}$ term only has to be dealt with when $u \in [-\epsilon/M,0] \cup [0,3\epsilon |\log(\epsilon)|]$.  
			
			Fix $\epsilon_{1}' > 0$ such that $\epsilon/M \leq 1/4$ if $\epsilon \in (0,\epsilon_{1}')$.  If $\epsilon \in (0,\epsilon_{1}')$ and $u \in [-\epsilon/M,0]$, then the definition of $M$ gives
				\begin{align*}
					f_{\epsilon}(u) = |f_{\epsilon}(u)| &= |\bar{f}(u)| \leq \theta^{-1} \|W''\|_{L^{\infty}([-3,3])} |u| = M |u|.
%
				\end{align*}
			Hence $f_{\epsilon}(u) \leq M |u| \leq \epsilon$, which gives
				\begin{equation*}
					\tilde{f}_{\epsilon}(u) = f_{\epsilon}(u) + \zeta_{\epsilon}(u) \left( \epsilon - f_{\epsilon}(u) - \frac{u}{|\log(\epsilon)|} \right) \geq f_{\epsilon}(u).
				\end{equation*}
			
			Making $\epsilon_{1}'$ smaller if necessary, we can assume that $3 \epsilon |\log(\epsilon)| \leq \mu$ if $\epsilon \in (0,\epsilon_{1}')$.  Now note that if $u \in [2\epsilon |\log(\epsilon)|,3 \epsilon |\log(\epsilon)|]$, then we can write
				\begin{align*}
					-\left(\frac{1}{|\log(\epsilon)|} + \frac{f_{\epsilon}(u) - f_{\epsilon}(0)}{u} \right)u &\geq -3 \epsilon - f_{\epsilon}(u) \\
					&\geq -3 \epsilon - \frac{W''(0)}{2 \Theta} u \\
					&\geq -3 \epsilon + \frac{|W''(0)| \epsilon |\log(\epsilon)| }{\Theta} 
				\end{align*}
			Finally, we let $\epsilon_{1}'' = \exp(-\frac{3 \Theta}{|W''(0)|})$ and $\epsilon_{1} = \epsilon_{1}' \wedge \epsilon_{1}'' \wedge \frac{1}{2}$.  The previous string of inequalities implies that if $\epsilon \in (0,\epsilon_{1} \wedge \epsilon_{0})$, then
				\begin{equation*}
					-\left(\frac{1}{|\log(\epsilon)|} + \frac{f_{\epsilon}(u) - f_{\epsilon}(0)}{u} \right)u \geq 0 \quad \text{if} \quad u \in [2\epsilon |\log(\epsilon)|,3 \epsilon |\log(\epsilon)|]
				\end{equation*}
		From this and \eqref{E: dreaded_formula}, it follows that $\tilde{f}_{\epsilon}(u) \geq f_{\epsilon}(u)$ for all $u \in [2 \epsilon |\log(\epsilon)|, 3 \epsilon |\log(\epsilon)|]$.
		
		Next, we note that if $u \in [0,2 \epsilon |\log(\epsilon)|]$ and $\epsilon \in (0,\epsilon_{1} \wedge \epsilon_{0})$, then a direct computation shows that
	\begin{align*}
		f_{\epsilon}(u) \leq - \frac{|W''(0)| u}{2 \Theta} \leq \epsilon - \frac{u}{|\log(\epsilon)|} = \tilde{f}_{\epsilon}(u)
	\end{align*}
	This completes the proof that $\tilde{f}_{\epsilon} \geq f_{\epsilon}$ in $[-2,2]$ and then the inequality $f_{\epsilon} \geq \bar{f}$ in the same interval follows from the construction of $f_{\epsilon}$.
	
	Next, we prove (b).  Recall that if $u \in [0,\mu]$, then
		\begin{equation*}
			f_{\epsilon}(u) \leq - \frac{|W''(0)|}{2 \Theta}u
		\end{equation*}
	and, thus, for all $u \in [2 \epsilon |\log(\epsilon)|,3 \epsilon |\log(\epsilon)|]$ and $\epsilon \in (0,\epsilon_{1} \wedge \epsilon_{0})$,
		\begin{align*}
			\tilde{f}_{\epsilon}(u) &\leq -(1 - \zeta_{\epsilon}(u)) \Theta^{-1} |W''(0)| \epsilon |\log(\epsilon)| - \zeta_{\epsilon}(u) \epsilon \leq - \epsilon
		\end{align*}
	
	Let $\epsilon_{1}''' = M\mu$.  If $u \in [-\frac{\epsilon}{M},0]$ and $\epsilon \in (0,\epsilon_{1}''')$, then
		\begin{align*}
			\tilde{f}_{\epsilon}(u) &\geq (1 - \zeta_{\epsilon}(u)) \frac{|W''(0)|}{2 \Theta} |u| + \zeta_{\epsilon}(u)\epsilon
		\end{align*}
	When $u \in [-\frac{\epsilon}{4M},0]$, this gives (by property (c) of $\zeta_{\epsilon}$ above), 
		\begin{equation*}
			\tilde{f}_{\epsilon}(u) \geq \left(1 - \frac{2M}{\epsilon} \cdot \frac{\epsilon}{4M} \right) \epsilon = \frac{\epsilon}{2}
		\end{equation*}
	while the case $u \in [-\frac{\epsilon}{M},-\frac{\epsilon}{4M}]$ yields
		\begin{equation*}
			\tilde{f}_{\epsilon}(u) \geq (1 - \zeta_{\epsilon}(u)) \frac{|W''(0)| \epsilon}{8M \Theta}+ \zeta_{\epsilon}(u)\epsilon \geq \frac{|W''(0)|}{8M \Theta} \cdot \epsilon
		\end{equation*}
	Therefore, if we replace $\epsilon_{1}$ above by $\epsilon_{1} \wedge \epsilon_{1}'''$, we conclude that there is a $c > 0$ such that (b) holds.

(c) follows directly from the choice of $\zeta_{\epsilon}$, conclusion (b) of Lemma \ref{L: regularization}, and \eqref{E: second_deriv_neg_side}.          \end{proof}  
	
Henceforth, we let $\chi^{\epsilon} : [-1 - \mu,1] \times [0,\infty) \to [-1 - \mu,1]$ denote the solution map of the ODE associated with $-\tilde{f}_{\epsilon}$, that is,
	\begin{equation*}
		\left\{ \begin{array}{r l}
				\chi^{\epsilon}_{s}(\xi,s) + \tilde{f}_{\epsilon}(\chi^{\epsilon}(\xi,s)) = 0 & \text{if} \, \, (\xi,s) \in [-1 - \mu,1] \times (0,\infty) \\
				\chi^{\epsilon}(\xi,0) = \xi & \text{if} \, \, \xi \in [-1 - \mu,1]
			\end{array} \right.
	\end{equation*}

	\begin{lemma} \label{L: universal ODE}  (i)  For each $\beta > 0$, there is an $\epsilon(\beta), \tau(\beta) > 0$ such that if $\epsilon \in (0,\epsilon(\beta))$, then
		\begin{equation} \label{E: epsilon_ODE}
			\chi^{\epsilon}(\xi,s) \geq 1 - \beta \epsilon \quad \text{if} \quad \xi \geq 3 \epsilon |\log(\epsilon)|, \, \, s \geq \tau(\beta)|\log(\epsilon)|
		\end{equation}
		
	(ii) $\chi^{\epsilon}_{\xi} > 0$ in $[-1 - \mu,1] \times [0,\infty)$, independently of $\epsilon > 0$.  
	
	(iii) For each $a > 0$, there is an $\epsilon(a) > 0$ and $B(a) > 0$ such that if $\epsilon \in (0,\epsilon(a))$, then
		\begin{equation} \label{E: magic_second_deriv_bound}
			\left| \frac{\chi_{\xi \xi}^{\epsilon}(\xi,s)}{\chi_{\xi}^{\epsilon}(\xi,s)} \right| \leq \frac{B(a)}{\epsilon} \quad \text{if} \quad (\xi, s) \in [-1 - \mu,1] \times [0,a |\log(\epsilon)|]
		\end{equation}
	\end{lemma}  
	
		\begin{proof}  The proof proceeds exactly as in \cite[Lemma 3.1]{chen}.  The main difference is $\tilde{f}''_{\epsilon}$ can grow like $C\epsilon^{-1}$ near $-1$, which, upon inspection of the proof in \cite{chen}, only has the effect of increasing the constant $B(a)$ in \eqref{E: magic_second_deriv_bound}. \end{proof}

\section{Comparison Principle} \label{A: comparison principle}  

After a multiplication by $m^{-1}$, \eqref{E: AC mobility} is a special case of the following class of equations:
	\begin{equation} \label{E: comparison}
		\left\{ \begin{array}{r l}
			u_{t} - G(x,Du) \text{tr}(D^{2}u) + B(x,u,Du) = 0 & \text{in} \, \, \mathbb{R}^{d} \times (0,T), \\
			u = u_{0} & \text{on} \, \, \mathbb{R}^{d} \times \{0\}.
			\end{array} \right.
	\end{equation}
In what follows, we assume that $G : \mathbb{R}^{d} \times \mathbb{R}^{d} \to (0,\infty)$ and $B : \mathbb{R}^{d} \times \mathbb{R}^{d} \to \mathbb{R}$ are bounded, continuous functions for which there are constants $C_{1}, K, m, M > 0$ such that, for each $(y,u,v) \in \mathbb{R}^{d} \times \mathbb{R} \times \mathbb{R}^{d}$, $y' \in \mathbb{R}^{d}$, and $u' \in \mathbb{R}$, we have
	\begin{align}
		|G(y,v) - G(y',v)| + |B(y,v,u) - B(y',v,u)| &\leq C_{1} \|y - y'\|, \\
		m \leq G(y,v) \leq M,& \\
		|B(y,u,v) - B(y,u',v)| &\leq K |u - u'|.
	\end{align}

As is standard in the theory of viscosity solutions, to obtain well-posedness of \eqref{E: comparison}, we start with a comparison principle:

\begin{theorem} \label{T: comparison} Fix $T > 0$.  If $u$ is a bounded, upper semi-continuous function satisfying $u_{t} - G(x,Du) \text{tr}(D^{2}u) + B(x,u,Du) \leq 0$ in $\mathbb{R}^{d} \times (0,T)$ and $v$ is a bounded, lower semi-continuous function satisfying $v_{t} - G(x,Dv) \text{tr}(D^{2}v) + B(x,v,Dv) \geq 0$ in $\mathbb{R}^{d} \times (0,T)$ and if $M$ is defined by 
	\begin{equation*}
		M= \lim_{\delta \to 0^{+}} \sup \left\{ u^{*}(x,0) - v_{*}(y,0) \, \mid \, x,y \in \mathbb{R}^{d}, \, \, \|x-y\| \leq \delta \right\},
	\end{equation*}
then
	\begin{equation*}
		u(x,t) - v(x,t) \leq M e^{Kt} \vee 0 \quad \text{for all} \, \, (x,t) \in \mathbb{R}^{d} \times (0,T).
	\end{equation*} 
\end{theorem}

	\begin{proof}[Sketch of the proof]  The Lipschitz assumption on $G$ implies that a comparison argument can be carried out in the spirit of \cite{user's guide}.  To get the exponential bound, we note that, given $\delta > 0$, if we define the functions $\tilde{u}^{\delta}$ and $\tilde{v}^{\delta}$ by $\tilde{u}^{\delta}(x,t) = e^{-(K + \delta) t} u(x,t)$ and $\tilde{u}^{-(K + \delta)t} v(x,t)$ and write $\tilde{w}^{\delta} = \tilde{u}^{\delta} - \tilde{v}^{\delta}$, then a standard argument shows that $\tilde{w}^{\delta}$ satisfies $\tilde{w}^{\delta} \leq M \vee 0$ in $\mathbb{R}^{d} \times (0,T)$.  The result is recovered after sending $\delta \to 0^{+}$.  \end{proof}  
	
Existence now follows using Perron's Method and regularization:

\begin{corollary} \label{C: perron}  Given $u_{0} \in BUC(\mathbb{R}^{d})$, there is a unique, bounded viscosity solution of \eqref{E: comparison}.  \end{corollary}  

	\begin{proof}  First, assume that $u_{0} \in BC^{2}(\mathbb{R}^{d})$.  It follows that $\bar{u}(x,t) = u_{0}(x) + Ct$ and $\underline{u}(x,t) = u_{0}(x) - Ct$ define super- and subsolutions of \eqref{E: comparison} provided $C > 0$ is large enough.  Thus, an application of Perron's Method gives a bounded, continuous solution $u$ with $u(\cdot,0) = u_{0}$ in $\mathbb{R}^{d}$.
	
	If $u_{0}$ is not so regular, then nonetheless we can find a sequence $(u_{0,n})_{n \in \mathbb{N}} \subseteq BC^{2}(\mathbb{R}^{d})$ such that $\|u_{0,n} - u_{0}\|_{L^{\infty}(\mathbb{R}^{d})} \to 0$ as $n \to \infty$.  The bound in Theorem \ref{T: comparison} implies that the associated solutions $(u_{n})_{n \in \mathbb{N}}$ are uniformly Cauchy in $\mathbb{R}^{d} \times [0,T]$.  Therefore, their limit $u = \lim_{n \to \infty} u_{n}$ exists and, by stability, is a solution of \eqref{E: comparison}.         \end{proof}  

\section{Construction of Correctors} \label{A: correctors}

In this section, we discuss the standing waves and correctors used in the ansatz \eqref{E: main ansatz}.  Throughout we assume that $W$ satisfies satisfies \eqref{A: zeros}, \eqref{A: sign of derivative}, and \eqref{A: nondegeneracy of W} as in Theorem \ref{T: main}.

\subsection{Standing Waves}  We begin by recalling some standard facts concerning the standing waves of the Allen-Cahn equation.  Up to a translation, this is the function $q : \mathbb{R} \to (-1,1)$ such that
	\begin{equation} \label{E: standing wave}
		-\ddot{q} + W'(q) = 0 \quad \text{in} \, \, \mathbb{R}, \quad \lim_{s \to \pm \infty} q(s) = \pm 1, \quad q(0) = 0.
	\end{equation}
	
	\begin{prop} \label{P: standing wave} There is a unique, strictly increasing $q : \mathbb{R} \to (-1,1)$ satisfying \eqref{E: standing wave}.    Further, there is a constant $C > 0$ such that
		\begin{equation*}
			|q(s) - 1| \leq C \exp \left(-\frac{s}{C} \right), \quad |q(s) + 1| \leq C \exp \left(\frac{s}{C} \right).
		\end{equation*}
	\end{prop}   
	
		\begin{proof}  \eqref{E: standing wave} has a Hamiltonian structure.  In particular, the expression $\frac{1}{2} \dot{q}(s)^{2} - W(q(s))$ is independent of $s \in \mathbb{R}$.  Since this should clearly be zero at infinity, we deduce that $\frac{1}{2} \dot{q}(0)^{2} = W(0)$.  Thus, $|\dot{q}(0)|$ is uniquely determined a priori.  We solve the ODE $\ddot{q} = W'(q)$ in $\mathbb{R}$ with $q(0) = 0$ and $\dot{q}(0) = \sqrt{2 W(0))}$, and then an exercise shows that the solution satisfies $\lim_{s \to \pm \infty} q(s) = \pm 1$.  Further, the identity $\dot{q}(s) = \sqrt{2 W(q(s))}$ implies $q$ is strictly increasing.
		
		The exponential convergence to $\pm 1$ can be proved using a stability analysis or the maximum principle.  Here we are using \eqref{A: sign of derivative}.    \end{proof}
		
The standing wave $q$ generates the solutions of \eqref{E: pulsating standing wave}.  More precisely, for each $e \in S^{d-1}$, the function $U_{e}(s,y) = q(s)$ is a solution of \eqref{E: pulsating standing wave} with $a \equiv \text{Id}$.  The penalized correctors constructed in Section \ref{S: penalized correctors} will be approximate solutions of \eqref{E: linearized equation}.  It will therefore be helpful to know some properties of the principal eigenfunction $\partial_{s} U_{e}$, which in this case equals $\dot{q}$.  

	\begin{prop}  $\dot{q} \in C^{2,\alpha}(\mathbb{R})$ solves the linearized Allen-Cahn equation:
		\begin{equation*}
			- \ddot{q} + W''(q(s)) \dot{q} = 0 \quad \text{in} \, \, \mathbb{R}.
		\end{equation*}
	Furthermore, there is a constant $C > 0$ such that $|\dot{q}(s)| \leq C \exp(-C^{-1}|s|).$		\end{prop}  
	
		\begin{proof}  The PDE is obtained directly by differentiating $q$.  The exponential convergence can be proved using the maximum principle and \eqref{A: sign of derivative} (cf.\ \cite[Proposition 26]{pulsating einstein}).\end{proof}  
		
	Finally, we will need the following fact concerning the drift appearing in the renormalized Allen-Cahn operator:
	
		\begin{prop} \label{P: good drift} The function $s \mapsto \frac{\ddot{q}(s)}{\dot{q}(s)}$ is bounded and uniformly Lipschitz continuous in $\mathbb{R}$.  \end{prop}  
		
			\begin{proof}  Applying Schauder estimates to $\dot{q}$, we find a constant $C > 0$ such that
				\begin{equation*}
					|\ddot{q}(s)| + |\dddot{q}(s)| \leq C ( \sup \left\{\dot{q}(s') \, \mid \, s' \in (s - 1, s + 1) \right\}).
				\end{equation*}
			Further, by the Harnack inequality, there is no loss of generality in assuming that
				\begin{equation*}
					\sup \left\{ \dot{q}(s') \, \mid \, s' \in (s - 1,s + 1) \right\} \leq C \dot{q}(s). 
				\end{equation*}
			Thus, $\frac{|\ddot{q}(s)|}{\dot{q}(s)} + \frac{|\ddot{q}(s)|}{\dot{q}(s)} \leq C^{2}$.  This gives the boundedness of $\frac{\ddot{q}}{\dot{q}}$ directly and the uniform Lipschitz continuity after differentiation.
			\end{proof}  
		
\subsection{Linearized Allen-Cahn Equation}  In the construction of approximate correctors, we used the linearized Allen-Cahn operator $-\partial^{2}_{s} + W''(q(s))$.  In particular, the next solvability result was used:
	
	\begin{prop} \label{P: 1D corrector} If $f \in C^{\alpha}(\mathbb{R})$ for some $\alpha \in (0,1)$, then there is a unique $P^{f} \in C^{2,\alpha}(\mathbb{R})$ and a unique $\overline{f} \in \mathbb{R}$ solving the PDE:
		\begin{equation*}
			- \ddot{P}^{f} + W''(q(s)) P^{f} = (f(s) - \overline{f})\dot{q}(s) \quad \text{in} \, \, \mathbb{R}, \quad P^{f}(0) = 0.
		\end{equation*}
	Furthermore, there is a $C > 0$ such that $|P^{f}(s)| \leq \|P^{f}\|_{L^{\infty}(\mathbb{R})} \exp(-C^{-1}|s|)$.  
	\end{prop}  
	
		\begin{proof}  The operator $\mathcal{L} = -\partial_{s}^{2} + W''(q(s))$ with domain $H^{2}(\mathbb{R})$ is self-adjoint in $L^{2}(\mathbb{R})$ with closed range.  Therefore, $\text{Ran}(\mathcal{L}) = \text{Ker}(\mathcal{L})^{\perp}$.  Since $\dot{q}$ is a positive eigenfunction, it follows that, for each $g \in H^{2}(\mathbb{R})$, 
			\begin{equation*}
				\langle \mathcal{L}g,g \rangle_{L^{2}(\mathbb{R})} = \int_{-\infty}^{\infty} \dot{h}(s)^{2} \dot{q}(s)^{2} \, d s,
			\end{equation*}
		where $h(s) = \dot{q}(s)^{-1} g(s)$.  From this, it follows that $\text{Ker}(\mathcal{L}) = \text{span}\{\dot{q}\}$.  
		
		Thus, if $f \in C^{\alpha}(\mathbb{R})$, then $f \dot{q} \in L^{2}(\mathbb{R})$ and there is a $P^{f} \in H^{2}(\mathbb{R})$ with $\mathcal{L} P^{f} = (f - \overline{f})\dot{q}$ provided $\overline{f} \in \mathbb{R}$ is given by
			\begin{equation*}
				\overline{f} = c_{W}^{-1} \int_{- \infty}^{\infty} f(s) \dot{q}(s)^{2} \, ds.
			\end{equation*}
		The local maximum principle implies $P^{f} \in L^{\infty}(\mathbb{R})$ and then Schauder estimates give $P^{f} \in C^{2,\alpha}(\mathbb{R})$.  Since $\dot{q}(s) \to 0$ exponentially as $s \to \pm \infty$, a maximum principle argument shows that $P^{f}$ does also.  \end{proof}  
		
\subsection{Penalized Correctors} \label{S: penalized correctors}  Finally, we prove the existence and regularity of the penalized correctors used in Section \ref{S: approximate correctors}.  Recall that these are the solutions $(P_{2}^{\delta})_{\delta > 0}$ of the penalized cell problem
	\begin{equation} \label{E: penalized cell problem appendix}
		\tilde{m}_{2}(s,y) \dot{q}(s) + \delta P_{2}^{\delta} + \mathcal{D}_{e}^{*} \mathcal{D}_{e} P^{\delta}_{2} + W''(q) P^{\delta}_{2} = 0 \quad \text{in} \, \, \mathbb{R} \times \mathbb{T}^{d}.
	\end{equation}
We search for a solution of the form $P^{\delta}_{2}(s,y) = V^{\delta}_{2}(s,y) \dot{q}(s)$, employing the so-called Doob-Legendre transform.  Plugging this ansatz into the equation, we find
that $P^{\delta}_{2}$ solves \eqref{E: penalized cell problem appendix} if and only if $V^{\delta}_{2}$ solves
	\begin{equation} \label{E: cell problem txm}
		\tilde{m}_{2}(s,y) + \delta V_{2}^{\delta} + \mathcal{D}_{e}^{*} \mathcal{D}_{e} V^{\delta}_{2} - \frac{2 \ddot{q}(s)}{\dot{q}(s)} \langle e, \mathcal{D}_{e} V^{\delta}_{2} \rangle = 0 \quad \text{in} \, \, \mathbb{R} \times \mathbb{T}^{d}.
	\end{equation}

It is clear that \eqref{E: cell problem txm} has a solution since it is degenerate elliptic and $f$ is bounded.  However, we will not proceed in this way.  

Instead, notice that $V^{\delta}_{2}$ is a classical solution of \eqref{E: cell problem txm} if and only if the one-parameter family of functions $(v_{\zeta}^{\delta})_{\zeta \in \mathbb{R}}$ determined by
	\begin{equation*}
		v_{\zeta}^{\delta}(x) = V^{\delta}_{2}(\langle x,e \rangle - \zeta, x)
	\end{equation*}
gives rise to solutions of the following family of PDE:
	\begin{equation} \label{E: physical cell problems}
		\tilde{m}_{2}(\langle x,e \rangle - \zeta,x) + \delta v^{\delta}_{\zeta} - \Delta v^{\delta}_{\zeta} - \frac{2 \ddot{q}(\langle x,e \rangle - \zeta)}{\dot{q}(\langle x,e \rangle - \zeta)} \langle e, Dv^{\delta}_{\zeta} \rangle = 0 \quad \text{in} \, \, \mathbb{R}^{d}.
	\end{equation}

Where regularity considerations are concerned, it is convenient to construct the solution $V^{\delta}_{2}$ of \eqref{E: cell problem txm} by studying the solutions of \eqref{E: physical cell problems}.  Here is the main result in that regard:

	\begin{theorem} \label{T: construction of correctors} If $\tilde{m}$ satisfies \eqref{E: smooth}, then, for each $\delta > 0$ and $\zeta \in \mathbb{R}$, there is a $v_{\zeta}^{\delta} \in C^{2,\alpha}(\mathbb{R}^{d})$ solving \eqref{E: physical cell problems}.  Furthermore, the map $(\zeta,x) \mapsto v_{\zeta}^{\delta}(x)$ is twice continuously differentiable with respect to $\zeta$ and there is a constant $C > 0$ independent of $\delta$ such that
		\begin{equation*}
			\|v_{\zeta}^{\delta}\|_{L^{\infty}(\mathbb{R}^{d})} + \left\| \frac{\partial v_{\zeta}}{\partial \zeta} \right\|_{C^{1,\alpha}(\mathbb{R}^{d})} + \left\| \frac{\partial^{2} v_{\zeta}}{\partial \zeta^{2}} \right\|_{C^{\alpha}(\mathbb{R}^{d})} \leq C (1 +\delta^{-1}).
		\end{equation*}
	\end{theorem}  
	
This leads immediately to a regularity result for $V^{\delta}$:

	\begin{corollary} \label{C: cell problem solution}  If $\tilde{m}$ satisfies \eqref{E: smooth}, then the unique viscosity solution $V_{2}^{\delta}$ of \eqref{E: cell problem txm} is in $C^{2,\mu}(\mathbb{R} \times \mathbb{T}^{d})$ and there is a $\delta$-independent constant $C > 0$ depending only on $f$ such that
		\begin{equation*}
			\left\| V_{2}^{\delta} \right\|_{C^{2,\alpha}(\mathbb{R} \times \mathbb{T}^{d})} \leq C(1+ \delta^{-1}).
		\end{equation*}
	Furthermore, making $C$ larger if necessary, we also have:
		\begin{equation*}
			\|P^{\delta}_{2}(s,\cdot)\|_{L^{\infty}(\mathbb{T}^{d})} + \|\partial_{s} P_{2}^{\delta}(s,\cdot)\|_{L^{\infty}(\mathbb{T}^{d})} \leq C (1 + \delta^{-1}) \exp \left(-C^{-1}|s|\right).
		\end{equation*}\end{corollary}  

\begin{proof}[Proof of Theorem \ref{T: construction of correctors}]  Since the drift $\frac{\ddot{q}}{\dot{q}}$ is bounded and uniformly Lipschitz continuous by Proposition \ref{P: good drift}, \eqref{E: physical cell problems} has a unique solution $v^{\delta}_{\zeta} \in C^{2,\mu}(\mathbb{R}^{d})$ and Schauder estimates give
	\begin{align*}
		\|v^{\delta}_{\zeta}\|_{L^{\infty}(\mathbb{R}^{d})} &\leq \|f\|_{L^{\infty}(\mathbb{R} \times \mathbb{T}^{d})} \delta^{-1}, \\
		\|v^{\delta}_{\zeta}\|_{C^{2,\mu}(\mathbb{R}^{d})} &\leq C' \left(\|v^{\delta}_{\zeta}\|_{L^{\infty}(\mathbb{R}^{d})} + \|\tilde{m}\|_{C^{\mu}(\mathbb{R} \times \mathbb{T}^{d})}\right) \leq C' \|\tilde{m}\|_{C^{\mu}(\mathbb{R} \times \mathbb{T}^{d})} (1 + \delta^{-1}).
	\end{align*}
By uniqueness, it is easy to see that $\zeta \mapsto v^{\delta}_{\zeta}$ is continuous with respect to the topology of local uniform convergence. 

Recall the functions $(\tilde{v}_{\zeta})_{\zeta \in \mathbb{R}}$ defined in \eqref{E: recentered solutions} by translation.  As pointed out above, these functions satisfy \eqref{E: recentered equations}.  Thus, employing the method of difference quotients, we deduce that the functions $(w_{\zeta})_{\zeta \in \mathbb{R}}$ defined by $w_{\zeta} = \frac{\partial \tilde{v}^{\delta}}{\partial \zeta}$ satisfy
	\begin{equation*}
		\langle D_{y} \tilde{m}(\langle x,e \rangle,x + \zeta e), e \rangle + \delta w_{\zeta} - \Delta w_{\zeta} + \frac{2 \ddot{q}(\langle x,e \rangle)}{\dot{q}(\langle x,e \rangle)} \langle e, Dw_{\zeta} \rangle = 0 \quad \text{in} \, \, \mathbb{R}^{d}.
	\end{equation*}
Similarly, the functions $(p_{\zeta})_{\zeta \in \mathbb{R}}$ given by $p_{\zeta} = \frac{\partial^{2} \tilde{v}_{\zeta}}{\partial \zeta^{2}}$ satisfy
	\begin{equation*}
		\langle D^{2}_{y} \tilde{m}(\langle x,e \rangle, x + \zeta e)e ,e \rangle + \delta p_{\zeta} - \Delta p_{\zeta} + \frac{2 \ddot{q}(\langle x,e \rangle)}{\dot{q}(\langle x,e \rangle)} \langle e, Dp_{\zeta} \rangle = 0 \quad \text{in} \, \, \mathbb{R}^{d}.
	\end{equation*}
Thus, since $D_{y}\tilde{m}$ and $D^{2}_{y}\tilde{m}$ are just as regular as $\tilde{m}$, there is a $C > 0$ such that
	\begin{equation*}
		\left\| \frac{\partial \tilde{v}_{\zeta}}{\partial \zeta} \right\|_{C^{2,\alpha}(\mathbb{R}^{d})} + \left\| \frac{\partial^{2} \tilde{v}_{\zeta}}{\partial \zeta^{2}} \right\|_{C^{2,\alpha}(\mathbb{R}^{d})} \leq C (1 + \delta^{-1}).
	\end{equation*}
Furthermore, if $\zeta, \zeta' \in \mathbb{R}$, then the H\"{o}lder regularity of $D^{2}_{y}\tilde{m}$ yields
	\begin{align*}
		\left\| \frac{\partial^{2} \tilde{v}_{\zeta}}{\partial \zeta^{2}} - \frac{\partial^{2} \tilde{v}_{\zeta'}}{\partial \zeta^{2}} \right\|_{L^{\infty}(\mathbb{R}^{d})} &\leq \delta^{-1} \|D^{2}\tilde{m}_{y}\|_{C^{\alpha}(\mathbb{R} \times \mathbb{T}^{d})} |\zeta - \zeta'|^{\alpha}.
	\end{align*}
These bounds readily carry over to $(v_{\zeta})_{\zeta \in \mathbb{R}}$, giving the desired estimates.  
\end{proof}  

We proceed with the proof of Corollary \ref{C: cell problem solution}:

\begin{proof}[Proof of Corollary \ref{C: cell problem solution}]  Define $V^{\delta}_{2} : \mathbb{R} \times \mathbb{T}^{d} \to \mathbb{R}$ by 
	\begin{equation*}
		V^{\delta}_{2}(s,y) = v^{\delta}_{\langle y,e \rangle - s}(y).
	\end{equation*}
An exercise shows this is well-defined.  Differentiating, we eventually find
$V^{\delta}_{2} \in C^{2,\alpha}(\mathbb{R} \times \mathbb{T}^{d})$.  A calculus exercise shows that $V^{\delta}_{2}$ is a solution of \eqref{E: cell problem txm}.    

The exponential convergence of $P^{\delta}_{2}$ follows from a maximum principle argument as in \cite[Proof of Proposition 30]{pulsating einstein}.  Using the fact that $\partial_{s} P^{\delta}_{2}$ satisfies a structurally similar linear PDE, we obtain a similar exponential estimate on $\partial_{s} P^{\delta}_{2}$.  \end{proof}

\section*{Acknowledgements}  

It is a pleasure to acknowledge P.E.\ Souganidis and W.M.\ Feldman for helpful discussions and encouragement.

\end{document}